\documentclass[10pt,reqno,a4paper]{amsart}

\let\oldtocsection=\tocsection

\let\oldtocsubsection=\tocsubsection

\let\oldtocsubsubsection=\tocsubsubsection

\renewcommand{\tocsection}[2]{\hspace{0em}\oldtocsection{#1}{#2}}
\renewcommand{\tocsubsection}[2]{\hspace{2em}\oldtocsubsection{#1}{#2}}
\renewcommand{\tocsubsubsection}[2]{\hspace{2em}\oldtocsubsubsection{#1}{#2}}

\usepackage{amsmath,amsthm,amsfonts,amssymb}

\usepackage{marginnote}

\usepackage{fancyhdr}

\usepackage{graphicx}

\usepackage{pdfpages}

\usepackage{tikz}
\usetikzlibrary{matrix}
\usetikzlibrary{cd}

\usepackage{xcolor}

\numberwithin{figure}{section}
\numberwithin{equation}{section}

\newtheorem{thm}{Theorem}[section]
\newtheorem*{thmnonum}{Theorem}

\newtheorem{defn}[thm]{Definition}

\newtheorem{lmm}[thm]{Lemma}

\newtheorem{prp}[thm]{Proposition}

\newtheorem{remark}[thm]{Remark}

\newcommand{\tei}{Teichm\"uller}
\newcommand{\qc}{quasiconformal}
\newcommand{\qd}{quadratic differential}
\newcommand{\idt}{of $\id$-type}

\newcommand{\const}{\operatorname{const}}

\newcommand{\id}{\operatorname{id}}
\newcommand{\Id}{\operatorname{Id}}

\renewcommand{\Re}{\operatorname{Re\,}}
\renewcommand{\Im}{\operatorname{Im\,}}
\renewcommand{\mod}{\operatorname{mod\,}}

\DeclareMathOperator*{\esssup}{ess\,sup}

\newcommand{\abs}[1]{\lvert #1 \rvert}

\newcounter{reminder}

\graphicspath{ {./images/}}

\title{Infinite-dimensional Thurston theory\\and transcendental dynamics I:\\infinite-legged spiders}
\author{Konstantin Bogdanov}

\begin{document}
	
\begin{abstract}
	We develop techniques that lay out a basis for generalizations of the famous Thurston's Topological Characterization of Rational Functions for an \emph{infinite} set of marked points and branched coverings of infinite degree. Analogously to the classical theorem we consider the Thurston's $\sigma$-map acting on a \tei\ space which is this time infinite-dimensional --- and this leads to a completely different theory comparing to the classical setting.
	
	We demonstrate our techniques by giving an alternative proof of the result by Markus F\"orster about the classification of exponential functions with the escaping singular value.
\end{abstract}

\maketitle
	
\tableofcontents

\addtocontents{toc}{\protect\setcounter{tocdepth}{1}}

\section{Introduction}

This is the first article out of four, prepared in order to publish the results of author's doctoral thesis. In the second and third papers we use the toolbox of the infinite-dimensional Thurston theory developed here to classify the transcendental entire functions that are compositions of a polynomial and the exponential for which their singular values escape on disjoint dynamic rays (the most general mode of escape). In the fourth article we investigate continuity of such families of functions with respect to potentials and external addresses and use the continuity argument to extend our classification to the case of escape on the (pre-)periodic rays.

\subsection*{Motivation from transcendental dynamics}
A part of complex dynamics that has been studied intensively over the last decades is the dynamics of transcendental entire functions. Good introductory references for the subject are \cite{ErLyu,DS-Transcendental}.

As in polynomial dynamics, for a transcendental entire function $f$ an important and natural concept is its \emph{escaping set} $I(f)=\{z\in\mathbb{C}: f^n(z)\to\infty\text{ as }n\to\infty\}$. It is proved in \cite{SZ-Escaping} for the exponential family $\{e^z+\kappa: \kappa\in\mathbb{C}\}$ and in \cite{RRRS} for more general families of functions (of bounded type and finite order) that the escaping set of every function in the family is organized in form of \emph{dynamic rays}.

Roughly speaking, this means that every point $z\in I(f)$ can be joined to $\infty$ by a unique simple arc contained in $I(f)$ (or is mapped to such arc after finitely many iterations), so that the iterates $f^n$ restricted to the arc are injective, continuous at $\infty$ (here we agree for the sake of convenience that $f(\infty)=\infty$), and converge to $\infty$ uniformly. About the escaping points that belong to a dynamic ray we say that they \emph{escape on rays}. This is the general mode of escape for many important families of transcendental entire functions.


In many cases the points escaping on rays can be described by their potential (or ``speed of escape'') and external address (or ``combinatorics of escape'', i.e.\ the sequence of dynamic rays containing the escaping orbit). This is analogous to the B\"ottcher coordinates for polynomials where the points in the complement of the filled Julia set are encoded by their potential and external angle (another more general way to introduce ``B\"ottcher coordinates'' for trancendental entire functions of bounded type is described in \cite{LasseParaSpace}).

One of the most fruitful directions in the study of holomorphic dynamical systems is the investigation of parameter spaces. The most famous example is the study of the Mandelbrot set which is the set of parameters $c$ in the complex plane for which the critical value $c$ of the polynomial $z^2+c$ does not escape. The simplest analogue in the transcendental world is the space of complex parameters $\kappa$, each associated to the function $e^z+\kappa$. When we consider parameters $\kappa$ for which the singular value $\kappa$ of $e^z+\kappa$ escapes, we, roughly speaking, investigate a part of the ``complement of the Mandelbrot set of the exponential family''. The results in this direction are contained in \cite{FRS,MarkusParaRays,MarkusThesis}

Now, we can ask a question. ($\star$) \emph{For a transcendental entire function of finite type (i.e.\ with finitely many singular values), is there in its parameter space an entire function such that its singular values escape on rays and have initially prescribed potentials and external addresses?}

In this article we set up a flexible machinery for proving this kind of results for general families of entire functions and apply it to reprove the well-known classification result for the exponential family \cite{MarkusThesis}.

A main tool is what one might call infinite-dimensional Thurston theory in dynamics. It is a generalization of the famous \emph{Thurston's Topological Characterization of Rational Functions} \cite{DH}. The subject is interesting not only as a tool but also in its own right and requires an independent exposition.

\subsection*{Infinite-dimensional Thurston theory}

One of the questions of big interest in dynamics is the interplay between topology and geometry. Or saying differently, when there exists a geometrical object corresponding in a certain sense to the underlying topological structure? And if it exists, is it unique?

A major result in this direction is proven by William Thurston.

\begin{thmnonum}[Topological Characterization of Rational Functions \cite{DH,HubbardBook2}]
	A post-critically finite branched covering $f:S^2\to S^2$ with hyperbolic orbifold is realized via a rational map if and only if there is no Thurston obstruction.
\end{thmnonum}

This result gives a complete answer to the question whether there exists a post-critically finite rational map with a prescribed combinatorics of the covering.

So it is natural to hope that if we prove an analogous ``characterization of entire functions'' for ``topological entire functions'' representing some properly chosen combinatorial model, then the corresponding entire function will have properties as required in question ($\star$). It turns out that this can be done (at least for certain explicit families of functions), but the proof, though keeping some similarities such as iterations on a \tei\ space, goes in quite a different direction than the proof of the classical theorem. The key differences between our and the classical case are that we have to consider iterations on an \emph{infinite-dimensional} \tei\ space (since the set of marked points is infinite), and that our topological entire function is a branched covering of \emph{infinite} degree. This drastically changes the approach already on the level of the setup. So far the case of infinite set of marked points for different type of functions was explored in \cite{BrownThesis,Cui,MarkusThesis}. A generalization of the Characterization Theorem for post-singularly finite topological exponential functions is provided in \cite{HSS}.

Before discussing details we present the ``grand scheme'' of how one can construct entire function demanded in question ($\star$). This ``grand scheme'' is partially inspired by the ideas in \cite{MarkusThesis}. 
\begin{enumerate}
	\item Construct a (non-holomorphic) ``model map'' for which the singular values escape on rays with desired ``speed of escape'' and ``combinatorics''. It will define (analogously to the classical case) the Thurston's $\sigma$-map acting on the \tei\ space of the complement of the singular orbits of the ``model map''.
	\item Construct a compact subspace in the \tei\ space that is invariant under $\sigma$.
	\item Prove that $\sigma$ is strictly contracting  in the \tei\ metric on the constructed subspace.
	\item Using the argument of the Banach Fixed Point Theorem prove that $\sigma$ has a unique fixed point on the constructed subspace, and this point corresponds to the entire function with the desired conditions on its singular orbits. 
\end{enumerate}

This scheme was successfully employed to prove classification results for compositions of polynomials with the exponential \cite{MyThesis}, and we believe that it works in a much greater generality (after adjustments of its internal ingredients). In this article we develop the techniques of how to treat every item (1)-(4) and use the scheme to answer the question ($\star$) in the ``baby case'' of the exponential family. Among all, the item (2) is the ``weightiest'' one, and we elaborate it in the two largest Sections~\ref{sec:spiders} and \ref{sec:inv_compact_subset}.

\subsection*{Setup of Thurston iteration}

This is what happens in the item (1) of the ``grand scheme'' and Section~\ref{sec:iteration_setup}.

In the parameter space that we consider we pick a function $f_0$ that has escaping points realizing the desired ``speeds of escape'' and ``combinatorics'' (from question ($\star$)). Based on $f_0$ we construct (by a postcomposition with a ``capture'') a special quasiregular function $f$ for which its singular orbits coincide with the corresponding escaping orbits of $f_0$, and which are equal to $f_0$ outside of some compact set. Moreover, its singular values escape on rays (defined analogously as for entire maps). This function $f$ plays the role of the topological branched covering in the classical theorem.

Based on $f$ we can construct the $\sigma$-map acting on the \tei\ space (Definition~\ref{defn:teich_space_main}) of $\mathbb{C}\setminus P_f$ where $P_f$ is the post-singular set of $f$ (i.e.\ the set of marked points). The map $\sigma:[\varphi]\to[\tilde{\varphi}]$ is defined in the usual way via Thurston diagram (Section~\ref{sec:iteration_setup}).

\begin{center}
	\begin{tikzcd}
		\mathbb{C},P_f \arrow[r, "{\tilde{\varphi}}"] \arrow[d, "f"]	& \mathbb{C},\tilde{\varphi}(P_f) \arrow[d, "g=\varphi\circ f\circ\tilde{\varphi}^{-1}"] \\
		\mathbb{C},P_f \arrow[r, "{\varphi}"] & \mathbb{C},\varphi(P_f)
	\end{tikzcd}
\end{center}
\vspace{0.5cm}

An attempt to define the Thurston's $\sigma$-map as in the classical theorem for a \emph{topological} (rather than quasiregular) exponential function would necessarily fail: the $\sigma$-map still can be defined on the level of isotopy classes of homeomorphisms, but since the post-singular set is infinite, the isotopy class of $\tilde{\varphi}\circ\varphi^{-1}$ relative $\varphi(P_f)$ might not necessarily contain a \qc\ map (see Subsection~\ref{subsec:top_exp_fails} for more details). Yet in this case $\sigma$ can be interpreted as a map between different \tei\ spaces, but we do not touch such constructions in this article.

In the article we set up the iteration for the case of exponential family, but it can be easily upgraded to \emph{any} family of transcendental entire functions.

\subsection*{$\Id$-type maps and spiders}

The goal of item (2) in the ``grand scheme'' is to construct a compact subset of $\mathcal{T}_f$ which is invariant under $\sigma$.

The space $\mathcal{T}_f$ with its isotopies relative to infinitely many points in $P_f$ is quite complicated, so one might want to take a look at simpler subspaces. These considerations lead us to the notion of $\id$-type maps (Definition~\ref{defn:id_type}). These are \qc\ maps that tend to identity as their argument tends to $\infty$ on the post-singular dynamic rays of $f$. An important property of these maps is that the images of post-singular dynamic rays of $f$ preserve their asymptotics near $\infty$, and this property is invariant under $\sigma$ (under proper normalization). So $\id$-type maps form a $\sigma$-invariant subspace of $\mathcal{T}_f$ that respects the ``combinatorics'' of $f$.

To every $\id$-type map one can associate an \emph{infinite-legged spider} and run the so-called spider-algorithm \cite{Spiders,MarkusThesis}. The spider which is really the driving concept of the proofs, is a convenient presentation of a point in the \tei\ space that provides a more tame description of its homotopy information (every point in the \tei\ space is described by the positions of images of $P_f$, which is moduli space information, and homotopy information). One can introduce certain equivalence relation on the set of spiders, called projective equivalence (Definition~\ref{defn:proj_equiv}), for which the following theorem holds (more precise version is Theorem~\ref{thm:W_define_teich_point}).  

\begin{thmnonum}
	Projective equivalence of spiders is \tei\ equivalence, i.e.\ two spiders are projectively equivalent if and only if the associated $\id$-type maps represent the same point in the \tei\ space.
\end{thmnonum}

This way we obtain a description of quite complicated isotopy types of homeomorphisms relative countably many marked points via a sequence of homotopy types of simple arcs, each relative just finitely many marked points --- this is a serious simplification. Another merit of the construction is that these homotopy types behave well under Thurston iteration (Theorem~\ref{thm:homotopy_type_under_pullback}).

We demonstrate how the construction works for the exponential family, but it is generalizable for functions admitting escape on rays.

\subsection*{Invariant compact subset}

Such subset is described as a list of conditions that stay invariant under $\sigma$. We give a very rough list at the moment, for more details we refer directly to Theorem~\ref{thm:invariant_subset}.

The compact invariant subset is the set of points represented by \qc\ maps $\varphi$ such that :
\begin{enumerate}
	\item $\varphi$ is \idt;
	\item for the marked points $P_f$ there exists an integer $N>0$ not depending on $\varphi$ such that the images under $\varphi$ of the first $N$ points of $P_f$ are inside of the disk $\mathbb{D}_\rho(0)$ with $\rho>0$ not depending on $\varphi$, and images of the other points of $P_f$ are outside of $\mathbb{D}_\rho(0)$;
	\item for very point $z\in P_f$ outside of $\mathbb{D}_\rho(0)$, $\varphi(z)$ is contained in a small disk $U_z$ around $z$ so that the disks $U_z$ are mutually disjoint; 
	\item inside of $\mathbb{D}_\rho(0)$ the distances between marked points are bounded from below;
	\item the isotopy type $\varphi$ relative points of $P_f$ outside of $\mathbb{D}_\rho(0)$ is ``almost'' the isotopy type of identity; 
	\item the isotopy type $\varphi$ relative points of $P_f$ inside of $\mathbb{D}_\rho(0)$ is not ``too complicated'', i.e.\ there are quantitative bounds on how many times the marked points ``twist'' around each other.
\end{enumerate}

Conditions (2),(3) separate the complex plane into two subsets: $\mathbb{D}_\rho(0)$ where we have not so much control on the behavior of $\varphi$ but have finitely many post-singular points, and the complement of $\mathbb{D}_\rho(0)$ where the homotopy information is trivial but we have infinitely many marked points.

Conditions (2)-(4) describe the position in the moduli space, while (5)-(6) encode homotopy information of the point in the \tei\ space.

It is natural to expect that these conditions define a compact subset, invariance is less obvious though.

We provide the construction for the exponential family. Generalizations for more than one singular value require additional conditions on the orbits. 

\subsection*{Strict contraction of the $\sigma$-map}

Item (3) of the ``grand scheme'' is where the infinite-dimensionality of the \tei\ space shows up.

The statement that $\sigma$ is strictly contracting in the classical theorem by Thurston is deduced from the well-known theorem of \tei\ claiming that the \qc\ representative minimizing the \tei\ distance on a surface of finite analytic type is associated to a unique normalized integrable quadratic differential. We consider a surface of infinite analytic type, and in our case this statement is false. 

We are interested in finding some other way to prove strict contraction of $\sigma$ on the \tei\ space. However, this seems to be not an easy task. Fortunately, we can find a way around. If we consider \emph{asymptotically conformal} points in the \tei\ space, then from the Strebel's Frame Mapping Theorem follows the existence of a unique associated normalized integrable quadratic differential, and we can prove strict contraction of $\sigma$ on an asymptotically conformal subspace.

In particular, the following theorem holds (more precise version is Theorem~\ref{thm:sigma_strictly_contracting}).

\begin{thmnonum}
	$\sigma$ is invariant and strictly contracting on the subset of asymptotically conformal points of the \tei\ space.
\end{thmnonum}

In the article we prove the theorem for compositions of polynomials with the exponential. For more general families exactly this method does not work, but it might be possible to deduce strict contraction directly from the construction of the invariant compact subset.

\subsection*{Classification Theorem}

As was mentioned earlier, the existence of a fixed point in the invariant compact subset follows by a simple argument of Banach Fixed Point Theorem (the subset is asymptotically conformal). We also show that the entire map corresponding to the constructed fixed point has singular values that escape on rays with the prescribed ``speed of escape'' and ``combinatorics''.

In particular, we reprove the well-known classification of exponential functions with singular values escaping on rays.

\begin{thm}[Classification Theorem for $e^z+\kappa$. \cite{MarkusThesis}]
	\label{thm:main_thm}
	Let $\underline{s}$ be an exponentially bounded external address that is not (pre-)periodic, and $t>t_{\underline{s}}$ be a real number. Then in the family $e^z+\kappa$ there exists a unique entire function $e^z+\kappa_{\underline{s},t}$ such that its singular value $\kappa_{\underline{s},t}$ escapes on rays, has potential $t$ and external address $\underline{s}$.
	
	Conversely, every function in the family $e^z+\kappa$ such that its singular value escapes on non-(pre-)periodic rays is one of these. 
\end{thm}

For the definitions of an exponentially bounded external address and $t_{\underline{s}}$ we refer the reader to Subsection~\ref{subsec:esc_dynamics}.

\section{Prerequisites}

\subsection{Escaping dynamics in the exponential family}
\label{subsec:esc_dynamics}

The simplest among the entire functions of finite type are functions having exactly one singular value. Every such function is equal to $A_1\circ\exp\circ A_2$ for some affine maps $A_1,A_2:\mathbb{C}\to\mathbb{C}$ and is conformally conjugate to a function of the form $e^z+\kappa$. Hence, to study dynamical properties of entire functions with one singular value, it is enough to consider the exponential family $e^z+\kappa$.

For the sake of brevity we adopt the following notation.
\begin{defn}[$\mathcal{N}$]
	Denote by $\mathcal{N}$ the family of transcendental entire functions of the form $e^z+\kappa$ for $\kappa\in\mathbb{C}$.
\end{defn}

In this section we assemble some preliminary results and definitions about the structure of the escaping set of functions in $\mathcal{N}$. We follow the exposition in \cite{SZ-Escaping} though use a slightly different notation that is more suitable for us.

It is proven in \cite{SZ-Escaping} for the exponential family, and in \cite{RRRS} for much more general families, that the escaping set of such functions is organized in form of \emph{dynamic rays}.

\begin{defn}[Ray tails]
	\label{dfn:ray_tail}
	Let $f$ be a transcendental entire function. A \emph{ray tail} of $f$ is a continuous curve $\gamma:[0,\infty)\to I(f)$ such that for every $n\geq 0$ the restriction $f^n|_\gamma$ is injective with $\lim_{t\to\infty}f^n(\gamma(t))=\infty$, and furthermore $f^n(\gamma(t))\to\infty$ uniformly in $t$ as $n\to\infty$.
\end{defn}

\begin{defn}[Dynamic rays, escape on rays, endpoints]
	\label{dfn:dynamic_ray}
	A \emph{dynamic ray} of a transcendental entire function $f$ is a maximal injective curve $\gamma:(0,\infty)\to I(f)$ such that $\gamma|_{[t,\infty)}$ is a ray tail for every $t>0$.
	
	If a point $z\in I(f)$ belongs to a dynamic ray, we say that the point \emph{escapes on rays}(because in this case every iterate of the point belongs to a dynamic ray). 
	
	If there exists a limit $z=\lim_{t\to 0}\gamma(t)$, then we say that $z$ is an \emph{endpoint of the dynamic ray} $\gamma$.  
\end{defn}

From now we assume that $f\in\mathcal{N}$. Let $r\in\mathbb{R}$ be such that the right half-plane $\mathbb{H}_r=\{z\in\mathbb{C}:\Re z>r\}$ does not contain the singular value of $f$. Then the preimage of $\mathbb{H}_r$ under $f$ has countably many path-connected components, called \emph{tracts}, having $\infty$ as a boundary point, so that for every tract $T$ the restriction $f|_T$ is a conformal isomorphism (strictly speaking these are tracts of $f^2$ so we slightly abuse the standard terminology). We denote by $T_n$ the tract contained in the infinite strip between the straight lines $\Im z=2\pi i(n-1/2)$ and $\Im z=2\pi i(n+1/2)$. This way we enumerate the tracts by integer numbers. Clearly, this numeration does not depend on a particular choice of $r$.

Let $\gamma:[0,\infty)\to I(f)$ be a ray tail of $f$. Then for every $n\geq 0$ there is $t_n\geq 0$ and $s_n\in\mathbb{Z}$ such that $f^n\circ\gamma|_{[t_n,\infty)}$ is contained in the tract $T_{s_n}$. Moreover, all $t_n$ except finitely many are equal to $0$.

\begin{defn}[External address]
	Let $z$ be a point escaping on rays. We say that $z$ has \emph{external address} $\underline{s}=(s_0 s_1 s_2 ... )$ where $s_n\in\mathbb{Z}$, if each $f^n(z)$ belongs to a ray tail contained in $T_{s_n}$ near $\infty$.
	
	In this case we also say that the dynamic ray (or ray tail) containing $z$ has external address $\underline{s}$. 
\end{defn}

It is clear that the external address does not depend on a particular choice of $\mathbb{H}_r$ in the definition of tracts.

On the set of external addresses we can consider the usual shift-operator $\sigma:(s_0 s_1 s_2 ... )\mapsto (s_1 s_2 s_3 ... )$.

Now we define the function which allows to characterize the ``speed of escape'' of points escaping on rays.

\begin{defn}[$F(t)$]
	Denote by $F:\mathbb{R^+}\to\mathbb{R^+}$ the function $$F(t):=e^t-1.$$
\end{defn}

What one should know about $F$ is that its iterates grow very fast, as shown in the next elementary lemma.

\begin{lmm}[Super-exponential growth of iterates]
	For every $t,k>0$ we have $F^n(t)/e^{n^k}\to\infty$ as $n\to\infty$.
\end{lmm}

Next definition relates the last two notions.

\begin{defn}[Exponentially bounded external address, $t_{\underline{s}}$]
	\label{dfn:exp_bdd_address}
	We say that the sequence $\underline{s}=(s_0 s_1 s_2 ... )$ is \emph{exponentially bounded} if there exists $t>0$ such that $s_n/{F^n(t)}\to 0$ as $n\to\infty$. The infimum of such $t$ we denote by $t_{\underline{s}}$.	
\end{defn}

Next statement claims that no other type of external addresses can appear for points escaping on rays.

\begin{lmm}[Only exponentially bounded $\underline{s}$ \cite{SZ-Escaping}]
	\label{lmm:only_exp_bdd}
	If $z\in I(f)$ escapes on rays and has external address $\underline{s}=(s_0 s_1 s_2 ... )$, then $\underline{s}$ is exponentially bounded.	
\end{lmm}

The key statement about the escaping set of functions in $\mathcal{N}$ is the following theorem which is an accumulation of different results (of interest for us) from \cite{SZ-Escaping}.

\begin{thm}[Escape on rays and asymptotic formula \cite{SZ-Escaping}]
	\label{thm:as_formula}
	Let $f\in\mathcal{N}$. Then for every exponentially bounded external address there exists a unique dynamic ray realizing it.
	
	Further, if $\mathcal{R}_{\underline{s}}$ is the dynamic ray having exponentially bounded external address $\underline{s}=(s_0 s_1 s_2 ... )$, and no strict forward iterate of $\mathcal{R}_{\underline{s}}$ contains the singular value of $f$, then $\mathcal{R}_{\underline{s}}$ can be parametrized by $t\in(t_{\underline{s}},\infty)$ so that 
	\begin{equation}
		\label{eqn:as_formula}
		\mathcal{R}_{\underline{s}}(t)=t+2\pi i s_0 + O(e^{-t/2}),
	\end{equation}
	and $$f^n\circ \mathcal{R}_{\underline{s}}=\mathcal{R}_{\sigma^n \underline{s}}\circ F^n.$$ Asymptotic bounds $O(.)$ for $\mathcal{R}_{\sigma^n\underline{s}}(t)$ are uniform in $n$ on every ray tail contained in $\mathcal{R}_{\underline{s}}$.
	
	Every escaping point is mapped after finitely many iterations either on a dynamic ray or to an endpoint of a ray. If the singular value of $f$ does not escape, then all dynamic rays are parametrized as above, and every escaping point escapes either on rays or as endpoints of rays.
\end{thm}

The theorem above allows us to define the notion of potential.

\begin{defn}[Potential]
	Let $f\in \mathcal{N}$ and assume that $z$ escapes on rays with escaping address $\underline{s}$. We say that $t$ is the potential of $z$ if $\abs{f^n(z)-F^n(t)-2\pi is_n}\to 0$ as $n\to\infty$.	
\end{defn}

From Theorem~\ref{eqn:as_formula} follows that every point escaping on rays (without singular values on forward iterates) has a well-defined potential, and different points on the same ray have different potentials.

\subsection{Quasiconformal maps}
Standard references are \cite{Ahlfors,LehtoVirtanen}. We start with a few preliminary definitions.

\begin{defn}[Quadrilateral]
	A \emph{quadrilateral} $Q(z_1,z_2,z_3,z_4)$ is a Jordan domain $Q$ together with a sequence $z_1,z_2,z_3,z_4$ of boundary points called vertices of the quadrilateral. Order of vertices agrees with the positive orientation with respect to $Q$. Arcs $z_1 z_2$ and $z_3 z_4$ are called $a$-sides, arcs $z_2 z_3$ and $z_4 z_1$ are called $b$-sides.
\end{defn}

Every quadrilateral $Q$ is conformally equivalent to the unique canonical rectangle with the length of $b$-sides equal to 1. For a quadrilateral $Q$, the length of the $a$-sides of the canonical rectangle is called a (conformal) \emph{modulus} of $Q$ and is denoted by $\mod Q$. 

\begin{defn}[Maximal dilatation]
	Let $U$ be a plane domain and $\psi$ be an orientation-preserving homeomorphism of $U$. The \emph{maximal dilatation} of $\psi$ is called the number
	
	\begin{center}
		$K(\psi)=\sup_{\overline{Q}\subset U}\frac{\mod \psi(Q)}{\mod Q}$,\\
	\end{center}
	where the supremum is taken over all quadrilaterals $Q$ contained in $G$ together with its boundary.
\end{defn}

Now we can define \qc\ maps.

\begin{defn}[Quasiconformal map]
	An orientation-preserving homeomorphism $\psi$ of a plane domain $U$ is called quasiconformal if its maximal dilatation $K(\psi)$ is finite. If $K(\psi)\leq K<\infty$, then $\psi$ is called $K$-quasiconformal.
\end{defn}

It is easy to show that the inverse of a $K$-quasiconformal mapping is $K$-quasiconformal, and the composition of a $K_1$-quasiconformal and $K_2$-quasiconformal mapping is $K_1 K_2$-quasiconformal.

We also provide the analytic definition of quasiconformal maps. It is equivalent to the previous one.

\begin{defn}[Quasiconformal map]
	A homeomorphism $\psi$ of a plane domain $U$ is quasiconformal if there exists $k<1$ such that
	
	\begin{enumerate}
		\item $\psi$ has locally integrable, distributional derivatives $\psi_z$ and $\psi_{\overline{z}}$ on $U$, and
		\item $\abs{\psi_{\overline{z}}} \leq k\abs{\psi_z}$ almost everywhere.
	\end{enumerate}
	
	Such $\psi$ is called $K$-quasiconformal, where $K=\frac{1+k}{1-k}$.
\end{defn}

Every quasiconformal map is determined by its Beltrami coefficient.

\begin{defn}[Beltrami coefficient]
	The function $\mu_\psi(z)=\psi_{\overline{z}}(z)/{\psi_z (z)}$ (defined a.e. on $U$) is called the \emph{Beltrami coefficient} of $\psi$.
\end{defn}

Providing the Beltrami coefficient is almost the same as giving a quasiconformal map. Suppose $\mu(z)$ is a measurable complex-valued function defined on a domain $U\subset\hat{\mathbb{C}}$ for which $\lvert\lvert\mu\rvert\rvert_{L^\infty}=k<1$. We can ask whether it is possible to find a quasiconformal map $\psi$ satisfying the Beltrami equation
\begin{equation}
	\label{eqn:Beltrami}
	\psi_{\overline{z}}(z)=\mu (z) \psi_z (z)
\end{equation}
where the partial derivatives $\psi_z (z)$ and $\psi_{\overline{z}}(z)$ are defined in the sense of distributions and are locally integrable.

The answer for $U=\hat{\mathbb{C}}$ is contained in Measurable Riemann Mapping Theorem.

\begin{thm}[Measurable Riemann Mapping Theorem \cite{Gardiner}]
	The Beltrami equation~\ref{eqn:Beltrami} gives a one-to-one correspondence between the set of quasiconformal homeomorphisms of $\hat{\mathbb{C}}$ that fix the points $0,1$ and $\infty$ and the set of measurable complex-valued functions $\mu$ on $\hat{\mathbb{C}}$ for which $\lvert\lvert\mu\rvert\rvert_{L^\infty}<1$. Furthermore, the normalized solution $\psi^\mu$ depends holomorphically on $\mu$.
\end{thm}

\subsection{Beurling-Ahlfors extension}

Next, we formulate two theorems that help us to reconstruct a \qc\ map on a half-plane from its boundary values given by a \emph{quasisymmetric function}.

\begin{defn}[Quasisymmetric function]
	\label{defn:quasisymmetry}
	Let $\xi$ be a continuous strictly increasing self-homeomorphism of the real line. We say that $\xi$ is \emph{$\rho$-quasisymmetric} if there exists a positive constant $\rho$ such that for every $t>0$ and every $x\in\mathbb{R}$ we have
	$$\frac{1}{\rho}\leq\frac{\xi(x+t)-\xi(x)}{\xi(x)-\xi(x-t)}\leq\rho.$$
\end{defn}

The first theorem basically says that the restriction of the \qc\ self-homeomorphism of the upper half plane is quasisymmetric.

\begin{thm}[Boundary values of homeomorphisms {\cite[Theorem~6.2]{LehtoVirtanen}}]
	Let $\varphi$ be a self-homeomorphism of the closure of the upper half-plane fixing $\infty$ such that for every quadrilateral $Q$ with boundary on the real line holds 
	$$\mod\varphi(Q)\leq K\mod Q$$ for some $K\geq 1$.
	
	Then for every $t>0$ and every $x\in\mathbb{R}$ holds
	$$\frac{1}{\lambda(K)}\leq\frac{\varphi(x+t)-\varphi(x)}{\varphi(x)-\varphi(x-t)}\leq\lambda(K),$$
	where $\lambda:[0,\infty]\to\mathbb{R}$ is a continuous function not depending on $\varphi$ so that $\lambda(1)=1$. 
\end{thm}

\begin{remark}
	Strictly speaking, the authors in \cite{LehtoVirtanen} prove the theorem for $K$-\qc\ maps of the upper half plane but exactly the same proof works in the setting as above.
\end{remark}

The second theorem says that quasisymmetric maps can be promoted to \qc\ maps with a sharp estimate on the maximal dilatation.

\begin{thm}[Beurling-Ahlfors extension \cite{BA}]
	\label{thm:BA-extension}
	Let $\xi:\mathbb{R}\to\mathbb{R}$ be a function that is $\rho$-quasisymmetric. Then it extends to a $K$-\qc\ self-homeomorphism of the upper half-plane with $K\leq\rho^2$. 
\end{thm}

\subsection{\tei\ theory and quadratic differentials}

Good references are \cite{Gardiner,HubbardBook1}. To avoid unnecessary technicality we give \emph{ad hoc} definitions that work specifically in our context. 

\begin{defn}[\tei\ space of $\mathbb{C}\setminus V$]
	\label{defn:teich_space_main}
	Let $V$ be a discreet subset of $\mathbb{C}$ without accumulation points in $\mathbb{C}$. The \emph{\tei\ space} of the Riemann surface $\mathbb{C}\setminus V$ is the set of quasiconformal homeomorphisms of $\mathbb{C}\setminus V$ modulo post-composition with an affine map and isotopy relative $V$.
\end{defn}

The points in the \tei\ space of $\mathbb{C}\setminus V$ are equivalence classes $[\varphi]$ of \qc\ homeomorphisms $\varphi$ of $\mathbb{C}\setminus V$. 

\begin{remark}
	A more standard definition of the \tei\ space on a Riemann surface involves isotopy relative the \emph{ideal boundary} rather than the topological boundary. For planar domains the two definitions are equivalent \cite{Gardiner}.
\end{remark}

Every \tei\ space can be equipped with the special metric.

\begin{defn}[\tei\ distance]
	Let $[\varphi_0],[\varphi_1]$ be two points in the \tei\ space of $\mathbb{C}\setminus V$. The \tei\ distance $d_T([\varphi_0],[\varphi_1])$ is defined as	
	$$\inf\limits_{\psi\in [\varphi_1\circ (\varphi_0)^{-1}]} \log K(\psi),$$
	that is, a lower bound of the logarithm of maximal dilatations of \qc\ maps that belong to $[\varphi_1\circ (\varphi_0)^{-1}]$.
\end{defn}

The \tei\ space of $\mathbb{C}\setminus V$ equipped with the \tei\ distance is complete metric space.

We also need the notion of holomorphic quadratic differential.

\begin{defn}[Holomorphic \qd]
	A \emph{holomorphic quadratic differential} on $\mathbb{C}\setminus V$ is a meromorphic function on $\mathbb{C}\setminus V$.
\end{defn}

We say that a holomorphic quadratic differential $q$ is \emph{integrable} if $q\in L^1(\mathbb{C})$ (or equivalently $q\in L^1(\mathbb{C}\setminus V)$). An integrable quadratic differential has either a simple pole, or a removable singularity at punctures (i.e.\ at points of $V$). Further, it has a pole of order 4 at infinity and at least two finite simple poles.

The next theorem is an important result of \tei\ theory.

\begin{thm}[\tei\ Uniqueness Theorem \cite{Gardiner}]
	\label{thm:teich_uniqueness}
	Suppose $\varphi_0$ is a \qc\ map of $\mathbb{C}\setminus V$ having Beltrami coefficient $k_0 \frac{\abs{q_0}}{q_0}$ with an integrable \qd\ $q_0$ and $0\leq k_0<1$. Let $\varphi$ be any other \qc\ map of $\mathbb{C}\setminus V$ such that $\varphi^{-1}\circ \varphi_0\in[\id]$ in the \tei\ space of $\mathbb{C}\setminus \varphi_0(V)$, and let $\mu_\varphi$ be the Beltrami coefficient of $\varphi$. Then either there exists a set of positive measure in $\mathbb{C}\setminus V$ on which $\abs{\mu_\varphi(z)}>k_0$, or $\mu_\varphi(z)=k_0 \frac{\abs{q_0}}{q_0}$ almost everywhere.
\end{thm}

In this article we need a special \emph{asymptotically conformal} subset of the \tei\ space.

\begin{defn}[Asymptotically conformal points \cite{Gardiner}]
	\label{defn:as_conformal}
	A point $[\varphi]$ in the \tei\ space of $\mathbb{C}\setminus V$ is called \emph{asymptotically conformal} if for every $\epsilon>0$ there is a compact set $C\subset\mathbb{C}\setminus V$ and a representative $\psi\in[\varphi]$ such that $\abs{\mu_\psi}<\epsilon$ a.e.\ on $(\mathbb{C}\setminus V)\setminus C$.
\end{defn}

Every asymptotically conformal point evidently contains a \emph{frame mapping}.

\begin{defn}[Frame mapping \cite{Gardiner}]
	\label{defn:frame_mapping}
	Let $\varphi_0$ an extremal representative of $[\varphi_0]$ (i.e.\ minimizing the \tei\ distance in $[\varphi_0]$) in the \tei\ space of $\mathbb{C}\setminus V$, and $K_0$ be its maximal dilatation. If there exists a quasiconformal map $\varphi\in[\varphi_0]$ such that $\esssup_{(\mathbb{C}\setminus V) \setminus C}|K_\varphi| <K_0$ where $C$ is a compact subset of $\mathbb{C}\setminus V$, we say that $ \varphi $ is a frame mapping of the equivalence class $[\varphi_0]$.
\end{defn}

Since every asymptotically conformal point contains a frame map, next theorem provides a presentation of the extremal representatives of the asymptotically conformal points.

\begin{thm}[The Frame Mapping Theorem \cite{Gardiner}]
	\label{thm:frame_mapping_theorem}
	If the point $[\varphi]$ in the \tei\ space of $\mathbb{C}\setminus V$ has a frame mapping, then it has an extremal representative $\varphi_0$ which has Beltrami coefficient $\mu_{\varphi_0}=k_0 \dfrac{|q_0|}{q_0}$ where $q_0$ is an integrable holomorphic quadratic differential with $\|q_0\|_{L^1}=1$ and $0\leq k_0<1$. Further, if $ 0< k_0<1 $, then $ q_0 $ is uniquely determined.
\end{thm}

\section{Setup of Thurston iteration}
\label{sec:iteration_setup}

\subsection{Thurston iteration}
\label{subsec:iteration_setup}
In this section we are going to introduce a procedure called Thurston iteration on a \tei\ space. It will be seen later that this procedure allows us to find entire functions with the desired combinatorics of escape of the singular values.

Let $f_0$ be a transcendental entire function of finite type, and $\xi:\mathbb{C}\to\mathbb{C}$ be a \qc\ map. Define the quasiregular function $f=\xi\circ f_0$ and assume that every singular point of $f$ (i.e.\ image of a singular point of $f_0$ under $\xi$) either escapes or is (pre-)periodic under iterations of $f$. Let $P_f$ be the union of all singular orbits of $f$ (including singular values). It is common to call $P_f$ either the \emph{post-singular} set or the set of \emph{marked points}.

As in the proof of Thurston's Topological Characterization of Rational Functions \cite{DH}, one is interested in existence of an entire map $g$ which is Thurston equivalent to $f$.

\begin{defn}[Thurston equivalence]
	We say that $f$ is \emph{Thurston equivalent} to an entire map $g$ if there exist two homeomorphisms $\varphi,\psi:\mathbb{C}\to\mathbb{C}$ such that
	\begin{enumerate}
		\item $\varphi=\psi$ on $P_f$,
		\item the following diagram commutes
		
		\begin{tikzcd}
			\mathbb{C},P_f \arrow[r, "{\psi}"] \arrow[d, "f"]	& \mathbb{C},\psi(P_f) \arrow[d, "g"] \\
			\mathbb{C},P_f \arrow[r, "{\varphi}"] & \mathbb{C},\varphi(P_f)
		\end{tikzcd}
		\item $\varphi$ is isotopic to $\psi$ relative $P_f$.
	\end{enumerate}
\end{defn}

Note that unlike in the classical theory we consider Thurston equivalence relative the \emph{infinite} set of marked points $P_f$.

\begin{defn}[$\mathcal{T}_f$]
	Denote by $\mathcal{T}_f$ the \tei\ space of $\mathbb{C}\setminus P_f$.
\end{defn}

Then the Thurston map $$\sigma:\mathcal{T}_f\to\mathcal{T}_f$$ is defined as follows. Let $[\varphi]$ be a point in $\mathcal{T}_f$  where $\varphi$ is quasiconformal. Then $\varphi\circ f$ defines a conformal structure on $\mathbb{C}$. By the Uniformization theorem there exists a map $\tilde{\varphi}$ that is a conformal homeomorphism from $\mathbb{C}$ with the new complex structure (given by $\varphi\circ f$) to the standard $\mathbb{C}$ (and it will be \qc\ with respect to the standard $\mathbb{C}$). Then by definition $\sigma[\varphi]:=[\tilde{\varphi}]$. It is easy to check that $\sigma$ is well-defined on the level of the \tei\ space. The procedure is encoded in the following diagram (called \emph{Thurston diagram}):

\begin{center}
	\begin{tikzcd}
		\mathbb{C},P_f \arrow[r, "{\tilde{\varphi}}"] \arrow[d, "f=\xi\circ f_0"]	& \mathbb{C},\tilde{\varphi}(P_f) \arrow[d, "g"] \\
		\mathbb{C},P_f \arrow[r, "{\varphi}"] & \mathbb{C},\varphi(P_f)
	\end{tikzcd}
\end{center}
The map $g=\varphi\circ f\circ\tilde{\varphi}^{-1}$ corresponding to the right vertical arrow is by construction a transcendental entire function.  As in the classical case, the $\sigma$-map is evidently continuous.

As in the classical paper [DH], we are looking for a fixed point of $\sigma$. It is easy to see that the following lemma holds.

\begin{lmm}[Thurston equivalence and fixed points of $\sigma$]
	If $[\varphi]=[\tilde{\varphi}]$, then the map $g=\varphi\circ f\circ\tilde{\varphi}^{-1}$ on the right hand side of the Thurston diagram is Thurston equivalent to $f$.
\end{lmm} 

\subsection{Captured exponential function}
\label{subsec:capture}

Since our primary goal is to prove the Classification Theorem for the exponential family using Thurston iteration, we have to choose the map $f$ properly. More precisely, let $f_0\in\mathcal{N}$ be a function with the non-escaping (finite) singular value $v$. Next, let $\{a_{n}\}_{n=0}^\infty$ be an orbit of $f_0$ that escapes on rays $\mathcal{R}_{n}$, and $R_{n}$ be the part of the dynamic ray $\mathcal{R}_{n}$ from $a_{n}$ to $\infty$ (i.e.\ the ray tail with the endpoint $a_n$).

Now we describe a construction of the \emph{capture}. It is a carefully chosen \qc\ homeomorphism of $\mathbb{C}$ that is equal to identity outside of a bounded set and mapping the singular value of $f_0$ to $a_0$. 

Fix a bounded Jordan domain $U\subset\mathbb{C}\setminus\bigcup_{n>0} R_{n}$ containing both $v$ and $a_{0}$. Define an isotopy $c_u:\mathbb{C}\to\mathbb{C}, u\in[0,1]$ through \qc\ maps, such that $c_0=\id$ on $\mathbb{C}$, $c_u=\id$ on $\mathbb{C}\setminus U$ for $u\in[0,1]$, and $c_1(v)=a_{0}$. We say that the \qc\ map $c=c_1$ is a capture. 

\begin{remark}
	Note that the choice of the capture is not unique, so we select one of them.
\end{remark}

Define the map $f:=c\circ f_0$. It is a \emph{quasiregular} map whose singular orbits coincide with $\{a_{n}\}_{n=0}^\infty$, the orbit of $a_0$ under $f_0$ (also $R_{n+1}=f(R_n)$). In this article we are mainly interested in the $\sigma$-map associated to the map $f$ constructed this way.

\subsection{Motivation of the ``quasiregular setup''} 
\label{subsec:top_exp_fails}

We need to say a few words in order to explain why we do not follow the classical setup of Thurston iteration but use a \emph{quasiregular} rather than \emph{topological} branched covering $f$ as in the Thurston's Characterization Theorem. The main trouble is that for post-singularly infinite topological branched coverings $f$ the isotopy classes $[\varphi]$ and $\sigma[\varphi]$ generally belong to different \tei\ spaces (which never happens when $P_f$ is finite). As an example one could simply consider the exponential.

\begin{lmm}[Classical definition of $\sigma$ does not work]
	Let $f(z)=e^z$, and let $\{a_n\}=\{f^{\circ n}(0)\}_{n=0}^\infty$ be the orbit of its singular value $0$. Then there exist orientation-preserving (not \qc) homeomorphisms $\varphi,\tilde{\varphi}:\mathbb{C}\to\mathbb{C}$ such that $\varphi\circ f\circ\tilde{\varphi}^{-1}$ is entire transcendental function but the isotopy class $[\varphi\circ\tilde{\varphi}^{-1}]$ (relative $\tilde{\varphi}(P_f)$) does not contain any \qc\ homeomorphisms.
\end{lmm} 
\begin{proof}
	Pick an orientation-preserving homeomorphism $\varphi$ such that $\varphi(\overline{z})=\overline{\varphi}(z)$, and $\varphi(a_n)=a_{2n}$. In particular, it preserves the real line.
	
	Let $\tilde{\varphi}$ be a homeomorphism such that $g=\varphi\circ f\circ\tilde{\varphi}^{-1}$ is entire --- its existence can be proven as earlier by Uniformization theorem --- and normalized so that $g(z)=e^z+\kappa$. Since $\varphi(0)=\varphi(a_0)=a_0=0$, we have $\kappa=0$. Further, we may assume that $\tilde{\varphi}$ also preserves the real line. Hence $\tilde{\varphi}(a_n)=a_{2n+1}$.
	
	Now, consider the orientation-preserving homomorphism $\varphi\circ\tilde{\varphi}^{-1}$, which preserves the real line. We want to prove that the isotopy class $[\varphi\circ\tilde{\varphi}^{-1}]$ does not contain any \qc\ maps.
	
	Let $\psi\in[\varphi\circ\tilde{\varphi}^{-1}]$ and consider the annuli
	$$A_k:=\{z\in\mathbb{C}: a_{k}<\abs{z}<a_{k+2} \}.$$
	Then the annulus $B_n=\psi(A_{2n+1})$ separates the points $0,a_{2n}$ from the points $a_{2n+2},\infty$, and this holds for all $n>0$.
	
	From the general theory of \qc\ maps (see e.g.\ \cite{LehtoVirtanen}) we know that there is a real constant $C>0$ so that
	$$\mod(B_n)\leq C\mod(A_{2n}).$$
	Hence $$\frac{\mod(\psi(A_{2n+1}))}{\mod(A_{2n+1})}=\frac{\mod(B_n)}{\mod(A_{2n+1})}\leq\frac{C \mod(A_{2n})}{\mod(A_{2n+1})}\to 0$$
	as $n\to\infty$ and $\psi$ cannot be \qc.
\end{proof}

The solution to this issue of ``jumping'' between different \tei\ spaces is to take a quasiregular function $f$ which appears with its own complex plane, i.e.\ exactly what we did in the setup of the Thurston iteration.

\section{Strict contraction of $\sigma$}
\label{sec:strict_contraction}

In this section we prove that the $\sigma$-map associated to the captured exponential function (defined as in Subsection~\ref{subsec:capture}) is strictly contracting in the \tei\ metric on the subset of asymptotically conformal points in the \tei\ space (Definition~\ref{defn:as_conformal}). Instead of showing it directly, we prove a row of more general lemmas which cover the case of the composition of a polynomial with the exponential: it will not require any essential additional efforts but we will be able to use this theorem in the subsequent articles which handle the infinite-dimensional Thurston theory in this case.

We work in the more general setting of Subsection~\ref{subsec:iteration_setup} (not with the captured exponential function). More precisely, let $f_0$ be a transcendental entire function of finite type, and $\xi:\mathbb{C}\to\mathbb{C}$ be a \qc\ map. Consider the quasiregular function $f=\xi\circ f_0$ and assume that every singular point of $f$ either escapes or is (pre-)periodic. In this case we have the map $\sigma:\mathcal{T}_{f}\to\mathcal{T}_{f}$ associated to $f$. We want to apply it to the asymptotically conformal points in $\mathcal{T}_f$.

\begin{lmm}[Strict contraction of $\sigma$]
	\label{lmm:strict_contraction}
	If points $[\varphi],[\psi]\in\mathcal{T}_{f}$ and their $\sigma$-images $\sigma[\varphi],\sigma[\psi]\in\mathcal{T}_{f}$ are asymptotically conformal, and $[\varphi]\neq[\psi]$, then
	$$d_T(\sigma[\varphi],\sigma[\psi])<d_T([\varphi],[\psi]).$$
\end{lmm}

\begin{proof}
	Assume that $\sigma[\varphi]\neq\sigma[\psi]$, otherwise the statement is trivial. Consider the joint Thurston diagram for maps $ \varphi $ and $ \psi $: 
	
	\begin{center}
		\begin{tikzpicture}
			\matrix (m) [matrix of math nodes,row sep=3em,column sep=4em,minimum width=2em]
			{
				\mathbb{C},\tilde{\psi}(P_f) & \mathbb{C},P_f & \mathbb{C},\tilde{\varphi}(P_f) \\
				\mathbb{C},\psi(P_f) & \mathbb{C}, P_f & \mathbb{C},\varphi(P_f)\\};
			\path[-stealth]
			(m-1-1) edge node [left] {$g$} (m-2-1)
			(m-1-3) edge node [right] {$h$} (m-2-3)
			(m-1-2) edge node [right] {$f$} (m-2-2)
			edge node [above] {$\tilde{\psi}$} (m-1-1)
			edge node [above] {$\tilde{\varphi}$} (m-1-3)
			(m-2-2) edge node [below] [above] {$\psi$} (m-2-1)
			edge node [above] {$\varphi$} (m-2-3)
			;
		\end{tikzpicture}	
		
	\end{center}
	
	Here $ g $ and $ h $ are entire maps. Since we are only interested in the Teichm\"uller distance between $ \sigma[\varphi] $ and $ \sigma[\psi] $, we can denote $ \eta:=\varphi\circ\psi^{-1} $ and consider the following wrapped diagram:
	
	\begin{center}
		
		\begin{tikzpicture}
			\matrix (m) [matrix of math nodes,row sep=3em,column sep=4em,minimum width=2em]
			{
				\mathbb{C},\tilde{\psi}(P_f) & \mathbb{C},\tilde{\varphi}(P_f) \\
				\mathbb{C},\psi(P_f) & \mathbb{C},\varphi(P_f)\\};
			\path[-stealth]
			(m-1-1) edge node [left] {$g$} (m-2-1)
			edge node [above] {$ \tilde{\eta} $} (m-1-2)
			(m-1-2) edge node [right] {$h$} (m-2-2)
			(m-2-1) edge node [below] {$ [\eta] $} (m-2-2)	
			;
		\end{tikzpicture}
	\end{center}
	
	Since $ g $ and $ h $ are holomorphic, the $L^\infty$-norm of the Beltrami differential of $ \eta $ is preserved under the pull-back through $g$. It follows immediately that
	$$d_T(\sigma[\varphi],\sigma[\psi])\leq d_T([\varphi],[\psi]).$$
	
	Next, observe that $ [\eta] $ is an asymptotically conformal point (as a composition of such classes). Then due to Frame Mapping Theorem~\ref{thm:frame_mapping_theorem} and \tei\ Uniqueness Theorem~\ref{thm:teich_uniqueness} there exists the unique map $\eta_0$ with the smallest maximal dilatation in $ [\eta] $, and its Beltrami coefficient has form $ \mu_{\eta_0}=k_0 \frac{|q_0|}{q_0} $ where $ q_0 $ is an integrable holomorphic quadratic differential on $ \mathbb{C}\setminus\psi(P_f) $ such that $\|q_0\|_{L^1}=1$. Let $ \tilde{\eta}_0 $ be a solution of the Beltrami equation with the Beltrami coefficient $\mu_{\eta\circ g}$. Then the Beltrami coefficient of $ \tilde{\eta}_0 $ will have form $ \mu_{\tilde{\eta}_0}=k_0 \frac{|q_0\circ g|}{q_0\circ g }\frac{|g'^2|}{g'^2} $.
	
	Since $\sigma[\eta]$ is asymptotically conformal (as a composition of asymptotically conformal points) in $ \sigma[\eta] $ there is a unique extremal map $ \tilde{\eta}_1 $ 
	and its Beltrami differential has the form $ \mu_{\tilde{\eta}_1}=k_1 \frac{|q_1|}{q_1}$ where $ q_1 $ is an integrable quadratic differential on $\mathbb{C}\setminus\tilde{\psi}(P_f)$ such that $\|q_1\|_{L^1}=1$, and $k_1\leq k_0$.
	
	If $ k_1=k_0 $, then due to \tei\ Uniqueness Theorem~\ref{thm:teich_uniqueness} we have $\mu_{\tilde{\eta}_1}=\mu_{\tilde{\eta}_0} $ and $ q_0\circ g \, g'^2/q_1\in\mathbb{R} $ in $\mathbb{C}\setminus \tilde{\psi}(P_f)$. This implies that $q_0\circ g\,g'^2/q_1\equiv\const$. But this means that $q_0$ cannot have more than one pole. Indeed, by Great Picard Theorem $g$ attains every value with at most one exception infinitely many times. Since $\abs{P_f}>1$ ($f_0$ either has more than one singular value, or belongs to the exponential family and its unique singular value cannot be a fixed point), $P_f$ contains at least one such value $z$. If $q_0$ has a pole at $z$, then $q_1$ has poles at infinitely many preimages of $z$ under $g$. This is impossible since at most finitely many of them belong to $P_f$. So $q_1$ can have at most one pole in $\mathbb{C}$. For an integrable holomorphic \qd\ on $\mathbb{C}\setminus\psi(P_f)$ this is possible only if $q_1=0$, and this contradicts to the choice of $q_1$ so that $\|q_1\|_{L^1}=1$. So $ k_1<k_0 $, and this means that $ \sigma $ is strictly contracting.		
\end{proof}

We want to use Lemma~\ref{lmm:strict_contraction} to show that if $f_0$ is a composition of a polynomial with the exponential and $\xi$ is asymptotically conformal as a map (i.e.\ its maximal dilatation tends to $0$ as the argument tends to $\infty$), then the $\sigma$-map is strictly contracting on the subset of asymptotically conformal points in $\mathcal{T}_f$. In order to do so, we need to show that $\sigma$ maps every asymptotically conformal point to another asymptotically conformal point. We do this in the following two lemmas.

If $\mu(z)$ is the Beltrami coefficient of a \qc\ map $\varphi$, then we denote $K_\varphi(z):=\frac{1+\mu(z)}{1-\mu(z)}$.

\begin{lmm}[Correction of the dilatation]
	\label{lmm:correction_of_dilatation}
	Let $ \varphi $ be a $ K $-quasiconformal automorphism of $ \mathbb{C} $ such that for some $R>0$ we have $\esssup_{\mathbb{C} \setminus \mathbb{D}_R(0)}|K_\varphi(z)| \leq 1+\epsilon$. Then there exists $\rho\geq R$ and a quasiconformal automorphism $ \varphi' $  such that $ \varphi\equiv\varphi' $ on $\mathbb{C}\setminus\mathbb{D}_{\rho}(0)$, $\esssup_{\mathbb{C}}|K_{\varphi'}(0)| \leq (1+2\epsilon)^3$, and $ \varphi $ and $ \varphi' $ are isotopic relative $ (\mathbb{C}\setminus\mathbb{D}_\rho(0)) \cup \{0\} $.
\end{lmm}

\begin{proof}
	Choose $ r>1 $ such that $ \varphi(\mathbb{D}_R(0)) $ is contained in a disk of diameter $ \delta<1 $ with respect to the hyperbolic metric on $ \varphi (\mathbb{D}_r(0)) $. Let $ Q $ be a quadrilateral with its sides on the boundary of $ \mathbb{D}_r(0) $, and $ M,M' $ be moduli of $ Q $ and $ \varphi(Q) $, respectively. After pre- and post-composition of $ \varphi|_{\mathbb{D}_r} $ with the canonical conformal maps of $ Q $ and $ \varphi(Q) $ onto the rectangles $ [0,M]\times[0,1] $ and $ [0,M']\times[0,1] $, respectively, we obtain a $K$-quasiconformal homeomorphism $ \chi $ of these rectangles (as quadrilaterals) whose maximal dilatation does not exceed $ 1+\epsilon $ outside of some set having image contained in a hyperbolic disk $ D $ of diameter $ \delta<1 $ centered at $ z_0 \in [0,M']\times[0,1] $. 
	
	Let $ a:=\inf_{z\in D}\Re z,\, b:=\sup_{z\in D} \Re z$, and $\gamma $ be the hyperbolic geodesic connecting the most left and the most right points of $ D $. Consider the partition of $[0,M']\times[0,1]$ into three rectangles $A,B,C$ such that $$A=[0,a]\times[0,1],$$
	$$B=[a,b]\times[0,1],$$
	$$C=[b,M']\times[0,1].$$
	Note that if $z\in[0,M']\times[0,1]$ then its distance to the boundary of $Q$ does not exceed $M'/2$. Hence using the standard estimates for the hyperbolic metric we obtain
	
	$$ 
	\mod B=b-a=M'\int_{a}^{b}\frac{1}{2\frac{M'}{2}}dx\leq M'\int_{\gamma}\lambda|dz| \leq M'\delta
	$$
	where $ \lambda $ is the hyperbolic density on $ [0,M']\times[0,1] $.
	
	Thus due to the estimates on the maximal dilatation of $ \chi $ and Gr\"otzsch inequality for rectangles we have:
	
	$$ \mod \varphi(Q)=M'=\mod A+\mod B+\mod C\leq$$
	$$(1+\epsilon)(\mod\chi^{-1}(A)+\mod\chi^{-1}(B))+M'\delta\leq$$
	$$(1+\epsilon)\mod Q+M'\delta=(1+\epsilon)M+M'\delta
	$$
	
	Hence 
	$$ \mod \varphi(Q)\leq \frac{1+\epsilon}{1-\delta}\mod Q $$
	and by making $r$ bigger (or equivalently by making $\delta$ smaller) we can choose the quasisymmetry constant (Definition~\ref{defn:quasisymmetry}) of $ \chi|_{\partial\mathbb{D}_r(0)} $ arbitrarily close to $1+\epsilon$. Choose it to be $ 1+2\epsilon $. Thus by Theorem~\ref{thm:BA-extension} the Beurling-Ahlfors extension $\mathcal{B}_\varphi$ of $\varphi|_{\partial\mathbb{D}_r(0)}$ would give us a quasiconformal homeomorphism of $\mathbb{D}_r(0)$ onto $\varphi(\mathbb{D}_r(0))$ with the maximal dilatation not exceeding $(1+2\epsilon)^2 $. Outside of $\varphi|_{\partial\mathbb{D}_r(0)}$ define $\mathcal{B}_\varphi$ to be equal to $\varphi$.
	
	If we choose a ``fat'' enough annulus $ \mathbb{D}_\rho(0)\setminus\overline{\mathbb{D}}_r(0) $ around $ \mathbb{D}_r(0) $ for some $\rho>r$, we can construct a quasiconformal map $ \theta' $ such that $ \theta'|_{\mathbb{C}\setminus\varphi(\mathbb{D}_\rho(0))} \equiv\id$, $ \theta'(\mathcal{B}_\varphi(0))=0$, and its maximal dilatation on $ \varphi(\mathbb{D}_\rho(0)) $ is not bigger than $1+2\epsilon $. Denote $ \varphi':=\theta'\circ \mathcal{B}_\varphi $. Then $ \varphi'(0)=0 $, $ \varphi'|_{\mathbb{C}\setminus\mathbb{D}_\rho}=\varphi $, its maximal dilatation does not exceed $ (1+2\epsilon)^3 $. That it is isotopic to $\varphi$ relative to the $ (\mathbb{C}\setminus\mathbb{D}_\rho(0)) \cup \{0\} $ follows from Alexander trick.
\end{proof}

Now we are ready to show that in the case of the composition of a polynomial with the exponential the asymptotically conformal subset is invariant under $\sigma$.

\begin{lmm}[Invariance of asymptotically conformal subset]
	\label{lmm:invariance_as_conf}
	Let $p$ be a non-constant polynomial, $f_0=p\circ\exp$ and $f=\xi\circ f_0$ where $\xi:\mathbb{C}\to\mathbb{C}$ is asymptotically conformal (as a map). Then the $\sigma$-image of every asymptotically conformal point is asymptotically conformal.
\end{lmm}
\begin{proof}
	Let $[\psi]\in\mathcal{T}_f$ be asymptotically conformal. We want to prove that $\sigma[\psi]\in\mathcal{T}_{f}$ is also asymptotically conformal. After solving Beltrami equation for $\mu_{\psi\circ\xi\circ p}$ we obtain $[\varphi]\in\mathcal{T}_{\exp(P_f)\cup\{0\}}$ (note that $p(0)\in P_f$) such that the following expanded Thurston diagram (with $g$ belonging to the exponential family and a polynomial $q$) is commutative.
	\begin{center}
		\begin{tikzpicture}
			\matrix (m) [matrix of math nodes,row sep=3em,column sep=4em,minimum width=2em]
			{
				\mathbb{C},P_f & \mathbb{C},\tilde{\psi}(P_f) \\
				\mathbb{C},\exp(P_f)\cup\{0\} & \mathbb{C},\varphi(\exp(P_f)\cup\{0\})\\
				\mathbb{C},P_f & \mathbb{C},\psi(P_f)\\};
			\path[-stealth]
			(m-1-1) edge node [left] {$\exp$} (m-2-1)
			edge node [above] {$ \tilde{\psi} $} (m-1-2)
			(m-1-2) edge node [right] {$g$} (m-2-2)
			(m-2-1) edge node [above] {$ \varphi $} (m-2-2)	
			edge node [left] {$\xi\circ p$} (m-3-1)
			(m-2-2) edge node [right] {$q$} (m-3-2)
			(m-3-1) edge node [above] {$ \psi $} (m-3-2)
			;
		\end{tikzpicture}
	\end{center}
	
	Since $[\psi]$ and $\xi$ are asymptotically conformal (the latter as a map) and the preimage of a neighborhood of $\infty$ under $p$ is a neighborhood of $\infty$, the point $[\varphi]\in\mathcal{T}_{\exp(P_f)\cup\{0\}}$ is also asymptotically conformal.
	
	If $\esssup_{\mathbb{C} \setminus \mathbb{D}_R}|K_\varphi| \leq 1+\epsilon$, then using Lemma~\ref{lmm:correction_of_dilatation} we can find a quasiconformal homeomorphism $\varphi'$ which is isotopic to $\varphi$ relative $0$ and the complement of a disk $\mathbb{D}_\rho(0)$ (basically we forget finitely many marked points except $0$ that are contained in $\mathbb{D}_\rho(0)$), and so that $\esssup_{\mathbb{C}}|K_{\varphi'}| \leq (1+2\epsilon)^3$. Consider its lift $ \tilde{\varphi}' $. It is isotopic to $\tilde{\psi}$ relative $P_f\setminus X$ where $ X $ is a finite subset of $P_f$.
	
	Next, we can always find a quasiconformal homeomorphism $ \theta$, equal to identity outside of some bounded set so that $\theta\circ\tilde{\varphi}'$ is isotopic to $\tilde{\psi}$ relative $P_f$. Further, we can also isotope $\theta\circ\tilde{\varphi}'$ in small neighborhoods of marked points in this bounded set so that the obtained map is conformal in these neighborhoods.
	
	Since $ \epsilon $ can be chosen arbitrarily small, lemma is proven.
\end{proof}

After choosing $f_0\in\mathcal{N}$ and a capture $c$ like in Subsection~\ref{subsec:capture}, from Lemmas~\ref{lmm:strict_contraction}, \ref{lmm:invariance_as_conf} immediately follows

\begin{thm}[$\sigma$ is strictly contracting on as.\ conf.\ subset]
	\label{thm:sigma_strictly_contracting}
	If $f$ is the captured exponential function (defined as in Subsection~\ref{subsec:capture}), then the $\sigma$-map is invariant and strictly contracting on the subset of asymptotically conformal points in $\mathcal{T}_f$.
\end{thm}

\section{$\Id$-type maps}

\label{sec:id_type_maps}

The purpose of this subsection is to introduce the so-called $\id$-type (or identity-type) maps which are \qc\ homeomorphisms of $\mathbb{C}$ satisfying certain conditions (see Definition~\ref{defn:id_type}). They have two important properties: they are almost equal to identity near $\infty$ on certain dynamic rays, and the set of points $[\varphi]\in\mathcal{T}_f$ that contain a map \idt\ in its equivalence class is invariant under $\sigma$. This allows to reduce Thurston iteration to the so-called Spider Algorithm \cite{MarkusThesis,Spiders}.

We work with the captured exponential function $f=c\circ f_0$ constructed in Subsection~\ref{subsec:capture} and start with a few definitions which take central position in the theory that we develop.

\begin{defn}[Standard spider]
	$S_0=\cup_{n\geq 0} R_{n}$ is called the \emph{standard spider}.
\end{defn}

\begin{defn}[$\Id$-type maps]
	\label{defn:id_type}
	A quasiconformal map $\varphi:\mathbb{C}\to \mathbb{C}$ is \idt\ if there is an isotopy $\varphi_u:\mathbb{C}\to\mathbb{C},\ u\in [0,1]$ such that ${\varphi_1=\varphi},\ \varphi_0=\id$ and $\abs{\varphi_u(z)-z}\to 0$ uniformly in $u$ as $S_0\ni z\to \infty$.
\end{defn}

\begin{defn}[$\Id$-type points in $\mathcal{T}_f$]
	We say that $[\varphi]\in\mathcal{T}_f$ is \idt\ if $[\varphi]$ contains an $\id$-type map.	
\end{defn}

Alternatively, we can replace $S_0$ in Definition~\ref{defn:id_type} by endpoints of $S_0$, i.e.\ $P_f$. In this case there are more $\qc$ maps identified as \idt. Nevertheless, it is not hard to show that they define the same subset \idt\ points in $\mathcal{T}_f$. Definition~\ref{defn:id_type} is simply more suitable for our purposes.

\begin{defn}[Isotopy \idt\ maps]
	We say that $\psi_u$ is an \emph{isotopy \idt\ maps} if $\psi_u$ is an isotopy through maps \idt\ such that $\abs{\psi_u(z)-z}\to 0$ uniformly in $u$ as $S_0\ni z\to \infty$.
\end{defn}

The identity is obviously \idt. Since $c^{-1}$ is equal to identity outside of a compact set, it is also \idt, as well as the composition $\varphi\circ c$ of the capture with any $\id$-type map.

Next theorem claims that the $\sigma$-map is invariant on the subset of $\id$-type points in $\mathcal{T}_f$.

\begin{thm}[Invariance of $\id$-type points]
	\label{thm:pullback_of_id_type}
	If $[\varphi]$ is \idt, then $\sigma[\varphi]$ is \idt\ as well.
	
	More precisely, if $\varphi$ is \idt, then there is a unique $\id$-type map $\hat{\varphi}$ such that $\varphi\circ f\circ\hat{\varphi}^{-1}$ is entire.
	
	Moreover, if $\varphi_u$ is an isotopy of $\id$-type maps, then the functions $\varphi_u\circ f\circ\hat{\varphi}_u^{-1}$ have the form $e^z+\kappa_u$ where $\kappa_u$ depends continuously on $u$.
\end{thm}
\begin{proof}
	First, since $\hat{\varphi}$ is defined up to a postcomposition with an affine map, and $S_0$ contains continuous curves joining finite points to $\infty$, uniqueness is obvious. We only have to prove existence.
	
	Since $c^{-1}$ is also \idt, $\varphi$ can be joined to $c^{-1}$ through an isotopy \idt\ maps: we can simply take in the correct order the concatenation of two isotopies to $\id$ (concatenation of $\varphi\sim\id$ and $\id\sim c^{-1}$).
	
	Let $\psi_u, u\in [0,1]$ be this isotopy with $\psi_0=c^{-1}$ and $\psi_1=\varphi$ and let $\tilde{\psi}_u$ be the unique isotopy of \qc\ homeomorphisms fixing $0,1$ and $\infty$ (or simply the solutions of the Beltrami equation~\ref{eqn:Beltrami} for $\mu_{\psi\circ f}$ normalized to fix $0,1,\infty$.) Then the maps $h_u:=\psi_u\circ f \circ \tilde{\psi}_u^{-1}$ will have the form 	
	$$h_u(z)=\beta_u e^{\alpha_u z}+\kappa_u$$ 	
	where $\alpha_u,\beta_u,\kappa_u$ depend continuously on $u$.
	
	Next define $\hat{\psi}_u(z):=\alpha_u\tilde{\psi}_u(z) +\log(\beta_u)$, where the branch of the logarithm is chosen so that $\hat{\psi}_0=\id$. In this case
	$$g_u(z)=\psi_u\circ f \circ \hat{\psi}_u^{-1}(z)=e^z+\kappa_u$$ 
	is the homotopy of entire functions such that $g_u\in\mathcal{N}$, $g_0=f_0$ and $\kappa_u$ is continuous in $u$ and bounded on $[0,1]$.	
	
	Now, let $z\in S_0$ and $w=f(z)=f_0(z)$. Since $\abs{\psi_u(w)-w}\to 0$ as $S_0\ni w\to\infty$ uniformly in $u$, we can write $\psi_u(w)=w+o(1)$.  
	
	Then	
	$$\hat{\psi}_u(z)=\log(\psi_u(w)-\kappa_u)=\log(f_0(z)-\kappa_u)=\log(e^z+\kappa_0-\kappa_u).$$
	Hence, due to continuity and boundedness of $\kappa_0-\kappa_u$, for $S_0\ni z\to\infty$ we have
	$$\abs{\hat{\psi}_u(z)-z}\to 0$$
	uniformly in $u$.
	
	The theorem is proven with the assignment $\hat{\varphi}:=\hat{\psi}_1$.
\end{proof}

Theorem~\ref{thm:pullback_of_id_type} particularly means that the notion of external address is preserved when we iterate $\id$-type maps, in a sense that under proper normalization, that is, when we choose the map $\hat{\varphi}$ out of its \tei\ equivalence class, the images of dynamic rays under $\hat{\varphi}$ preserve their original asymptotics (horizontal straight lines). 

\begin{remark}
	In the sequel we will use the hat-notation from Theorem~\ref{thm:pullback_of_id_type}. That is, $\hat{\varphi}$ denotes the unique $\id$-type map so that $\varphi\circ f\circ\hat{\varphi}^{-1}$ is entire. Due to Theorem~\ref{thm:pullback_of_id_type} this notation makes sense whenever $\varphi$ is \idt.	
\end{remark}

\section{Spiders}
\label{sec:spiders}

In this section we introduce the notion of a spider and a special tool which is supposed to bound the complexity of ``spider legs'' --- altogether this allows to provide numerical estimates on the homotopy information of $\id$-type points in the \tei\ space. We start with the latter.

\subsection{Curves in the punctured half-plane}
\label{subsec:r_curves}

We begin with the definition that slightly extends the notion of the homotopy of curves fixing endpoints.

\begin{defn}[Pinched endpoints homotopy]
	\label{defn:pinched_endpts_homotopy}
	Let $U\subset\hat{\mathbb{C}}$ be a path-connected domain, and $\gamma:(0,\infty)\to U$ be a curve such that $\lim\limits_{t\to 0}\gamma(t)\in \overline{U}$ and $\lim\limits_{t\to\infty}\gamma(t)\in\overline{U}$.
	
	We denote by $[\gamma]_U$ the equivalence class of $\gamma$ in $U$ under homotopies preserving the limits at $t=0$ and $t=\infty$. 
\end{defn}

If the limit of $\gamma$ at $0$ or $\infty$ belongs to $U$, this is the usual homotopy fixing endpoint, but when the limit is in the boundary of $U$, the homotopies are considered through curves in $U$ that are ``pinched'' at the endpoint (possibly two) in the boundary.

We consider such ``pinched'' homotopy types for two reasons: first, it describes the homotopy type of the images of a ray tail (from $S_0$) under $\id$-type maps representing the same point in the \tei\ space; second, as will be seen later, we obtain estimates on the ``complexity'' of the homotopy information of $\id$-type points.

Now, we define \emph{r-curves} (``r'' means right) which model ray tails (or more precisely, their union with $\infty$).

\begin{defn}[r-curve]	
	We say that a continuous curve $\gamma:[0,\infty]\to\hat{\mathbb{C}}$ is an \emph{r-curve} if
	\begin{enumerate}
		\item $\gamma[0,\infty)\subset\mathbb{C}$,
		\item $\gamma(\infty)=\infty$,
		\item $\Re \gamma(t)\to+\infty$ as $t\to\infty$,
	\end{enumerate}
\end{defn}

Let $V=\{v_i\}\subset\mathbb{C}$ be a finite or a countably infinite sequence (of punctures) such that $\Re v_i\to +\infty$,  and  let $\gamma$ be an r-curve such that $\gamma(0)\in V$ and $\gamma|_{\mathbb{R}^+}\subset\mathbb{C}\setminus V$.

\begin{remark}
	In the case when $V$ is infinite the surface $\mathbb{C}\setminus V$ is supposed to model the surface $\mathbb{C}\setminus P_f$.
\end{remark}

Choose a point $r\in\mathbb{R}$ such that 
$$r < \min \{\min\limits_{v\in V}\{\Re v\},\min\limits_{t\in [0,\infty]}\{\Re \gamma(t)\}\}. $$ 
Both $V$ and $\gamma|_{\mathbb{R}^+}$ are contained in $\mathbb{H}_r$ while $\gamma(\infty)=\infty\in\partial\mathbb{H}_r$.

We want to have an algebraic description of $[\gamma]_{\mathbb{H}_r\setminus V}$ (or more precisely $[\gamma|_{\mathbb{R}^{+}}]_{\mathbb{H}_r\setminus V}$), for instance, in terms of certain fundamental group. But the problem is that the endpoints of $\gamma$ are punctures, that is, the homotopy is not allowed  to pass through them. We cure this problem as follows.

Let 
$$H_r=H_r(V,\gamma):=(\mathbb{H}_r\setminus V)\cup\{\infty\}.$$

For a curve $\delta:(0,\infty)\to\mathbb{C}\setminus V$ with $\lim_{t\to 0}\delta(t)\in V$ and $\lim_{t\to \infty}\delta(t)=\infty$ let $\overline{\delta}:(0,\infty]\to H_r$ be the extension of this curve to $\infty$, that is, $\overline{\delta}(t)=\delta(t)$ for $t\in(0,\infty)$ and $\overline{\delta}(\infty)=\infty$.

\begin{lmm}
	\label{lmm:bar_paths_equivalence}
	For curves $\delta:(0,\infty)\to\mathbb{C}\setminus V$ with $\lim\limits_{t\to 0}\delta(t)\in V$ and $\lim\limits_{t\to \infty}\delta(t)=\infty$ the correspondence $[\delta]_{\mathbb{H}_r\setminus V}\mapsto[\overline{\delta}]_{H_r}$ is well defined and injective.
\end{lmm}
\begin{proof}
	That the correspondence is well defined follows from Definition~\ref{defn:pinched_endpts_homotopy}.
	
	Thus, we only need to prove that the curves $\delta_1$ and $\delta_2$ are homotopic in $\mathbb{H}_r\setminus V$ if $\overline{\delta}_1$ and $\overline{\delta}_2$ are homotopic in $H_r$.
	
	Let $v\in V$ be the puncture at which start both $\overline{\delta}_1$ and $\overline{\delta}_2$. Let also 
	$$\Gamma:[0,\infty]\times[0,1]\to H_r\cup\{v,\infty\}$$
	with $\Gamma(t,0)=\overline{\delta}_1$ and $\Gamma(t,1)=\overline{\delta}_2$ be a homotopy between $\overline{\delta}_1$ and $\overline{\delta}_2$ in $H_r$, that is, for $u\in[0,1]$ we have $\Gamma(0,u)=v$ and $\Gamma(\infty,u)=\infty$, and no inner point of the strip $B:=[0,\infty]\times[0,1]$ is mapped to $v$ (but can be mapped to $\infty$).
	
	Without loss of generality we might assume that $\Gamma([0,1]\times[0,1])$ is bounded, otherwise just rescale $\Gamma$ in the first parameter. For every integer $n>0$ let $\beta_n$ be the restriction of $\Gamma$ on $n\times[0,1]$, that is, a trajectory of a point $\overline{\delta}_1(n)$ under the homotopy. Since the image under $\Gamma$ of the boundary of the rectangle $[0,n]\times[0,1]$ is contractible, each $\beta_n$ is homotopic relative endpoints to a curve $\beta'_n$ contained in $H_r\setminus\{\infty\}$. Moreover, we can make a joint choice of all $\beta'_n$ so that $\beta'_n\to\infty$ as $n\to\infty$ (that is, the whole curve $\beta'_n$ appears in arbitrarily small neighborhoods of $\infty$ in $H_r$ when $n$ is big).
	
	Now, let $\alpha_n:\partial\mathbb{D}\to H_r\setminus\{\infty\}$ be the closed curve in $H_r\setminus\{\infty\}$ that is equal to the concatenation $$\beta_n^{-1}\cdot\overline{\delta}_1|_{[n,n+1]}\cdot\beta_{n+1}\cdot(\overline{\delta}_2|_{[n,n+1]})^{-1}.$$
	Each $\alpha_n$ is contractible in $H_r\setminus\{\infty\}$, so one can continuously extend $\alpha_n$ to the unit disk and obtain a map $\alpha'_n:\overline{\mathbb{D}}\to H_r\setminus\{\infty\}$ that coincides with $\alpha_n$ on $\partial\mathbb{D}$. Moreover, since $\beta'_n\to\infty$, we can choose $\alpha'_n$ so that $\alpha'_n(\overline{\mathbb{D}})\to\infty$. After ``gluing'' all $\alpha'_n$ and the restriction $\Gamma|_{[0,1]\times[0,1]}$ together we obtain the desired homotopy.
	
	This finishes the proof of the lemma.
\end{proof}

\begin{figure}[h]
	\includegraphics[width=\textwidth]{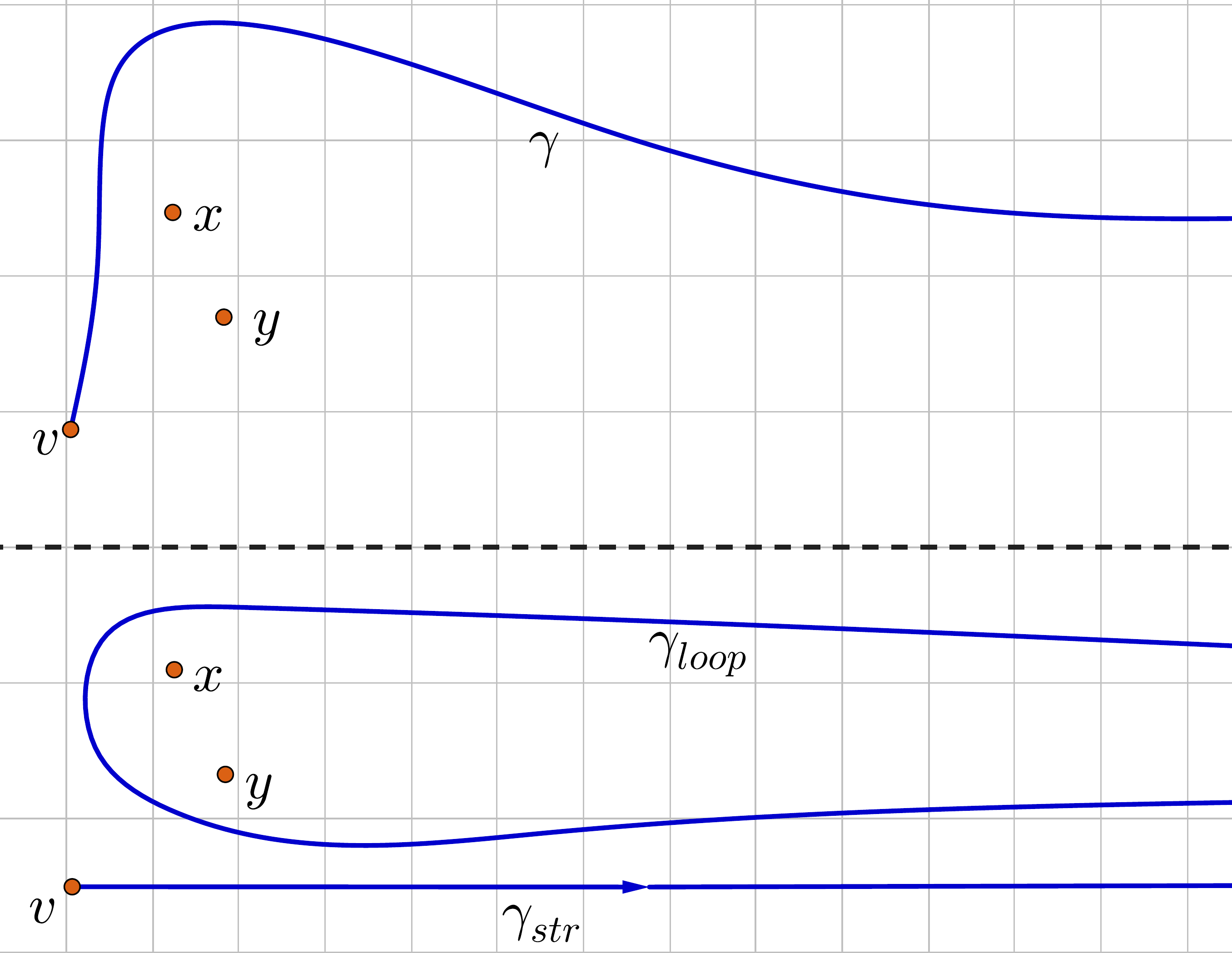}
	\caption{Example of splitting $\gamma$ (above) into $\gamma_{str}$ and $\gamma_{loop}$}
	\label{pic:split}
\end{figure} 

Now we are ready to describe the homotopy type of $\gamma$. In the sequel we only  consider the homotopy types in $H_r$ and denote $[\gamma]=[\gamma]_{H_r}$.

Join $\gamma(0)$ to $\infty$ by the straight horizontal ray $\langle\gamma(0)+t, t>0\rangle$. For every intersection of this ray with some point $v\in V$, choose a small closed disk with center at $v$ such that all disks are mutually disjoint and contain exactly one point of $V$ (its center), and replace the parts of the straight line which are the diameters of the disks by the upper half-circles of the disk boundary. 

The obtained curve $\gamma_{str}$ defines the homotopy class $[\gamma_{str}]$ in $H_r$, containing curves that approximate ``from above'' the straight horizontal ray from $\gamma(0)$ to $\infty$. Consider also a loop $\gamma_{loop}\subset H_r$ with the base point at $\infty$ such that $[\gamma_{str}]\cdot [\gamma_{loop}]=[\gamma]$ in $H_r$. It is not hard to show that such a loop exists: for example, one can slightly homotope $\gamma$ in $H_r$ so that for some $\epsilon>0$ we have $\gamma|_{(0,\epsilon]}(t)=\gamma(0)+t$ --- then $[\gamma_{loop}]=[(\gamma_{str}|_{[\epsilon,\infty]})^{-1}]\cdot[\gamma|_{[\epsilon,\infty]}]$. The loop homotopy type $[\gamma_{loop}]$ can be viewed as an element of the fundamental group $\pi_1(H_r,\infty)$. The homotopy class $[\gamma_{loop}]$ and the point $\gamma(0)$ uniquely define $[\gamma]$.

From now on we assume that $V$ is finite. Any such $V$ and an r-curve $\gamma$ with $\gamma(0)\in V$ and $\gamma|_{(0,\infty)}\in\mathbb{C}\setminus V$ define the map  

$$W: (V,\gamma)\mapsto W(V,\gamma)\in F(V),$$
where $F(V)$ is the free group on $V$, as follows. 

For every $w\in V$ construct a generator $g_w$ of $\pi_1(H_r,\infty)$ in the following way. Surround each $v\in V$ lying on the straight horizontal ray $\langle w+t, t\geq0\rangle$ by a disk $D_v$ with center at $v$ and radius $r_v$, so that the disks are mutually disjoint and contain only one point of $V$. Define the injective curve $a:[0,\infty]\to H_r$ so that $a(t)$ is equal to $w+r_w+t$ when $w+r_w+t$ is outside of the disks $D_v$ and coincides with the upper half-circles of $\partial D_v$ whenever $w+r_w+t$ passes through $D_v$ for $v\neq w$, and let $c(t):=w+r_u e^{-2\pi i t}$ for $t\in [0,1]$. Then $g_w:=[a^{-1}\cdot c\cdot a]$, and it does not depend on a particular choice of radii of the disks $D_v$. In other words, we start at $\infty$, move ``to the left'' along a straight line passing ``from above'' the points $v\in V$ lying in our way, then make a loop around $w$ in the clockwise direction, and return to $\infty$ the along the same curve, passing occasional points in $V$ ``from above''.

With this preferred set of generators the homotopy type $[\gamma_{loop}]$ has the unique representation in $\pi_1(H_r,\infty)$. That is, $[\gamma_{loop}]=g_{v_1}g_{v_2}\dots g_{v_m}$ where $v_i \in V, m\in \mathbb{N}$. Define $$W(V, \gamma):=v_1 v_2\dots v_m. $$ It is clear that $W$ does not depend on a particular choice of $H_r$ (i.e.\ the choice of $r$). 

\begin{figure}[h]
	\includegraphics[width=\textwidth]{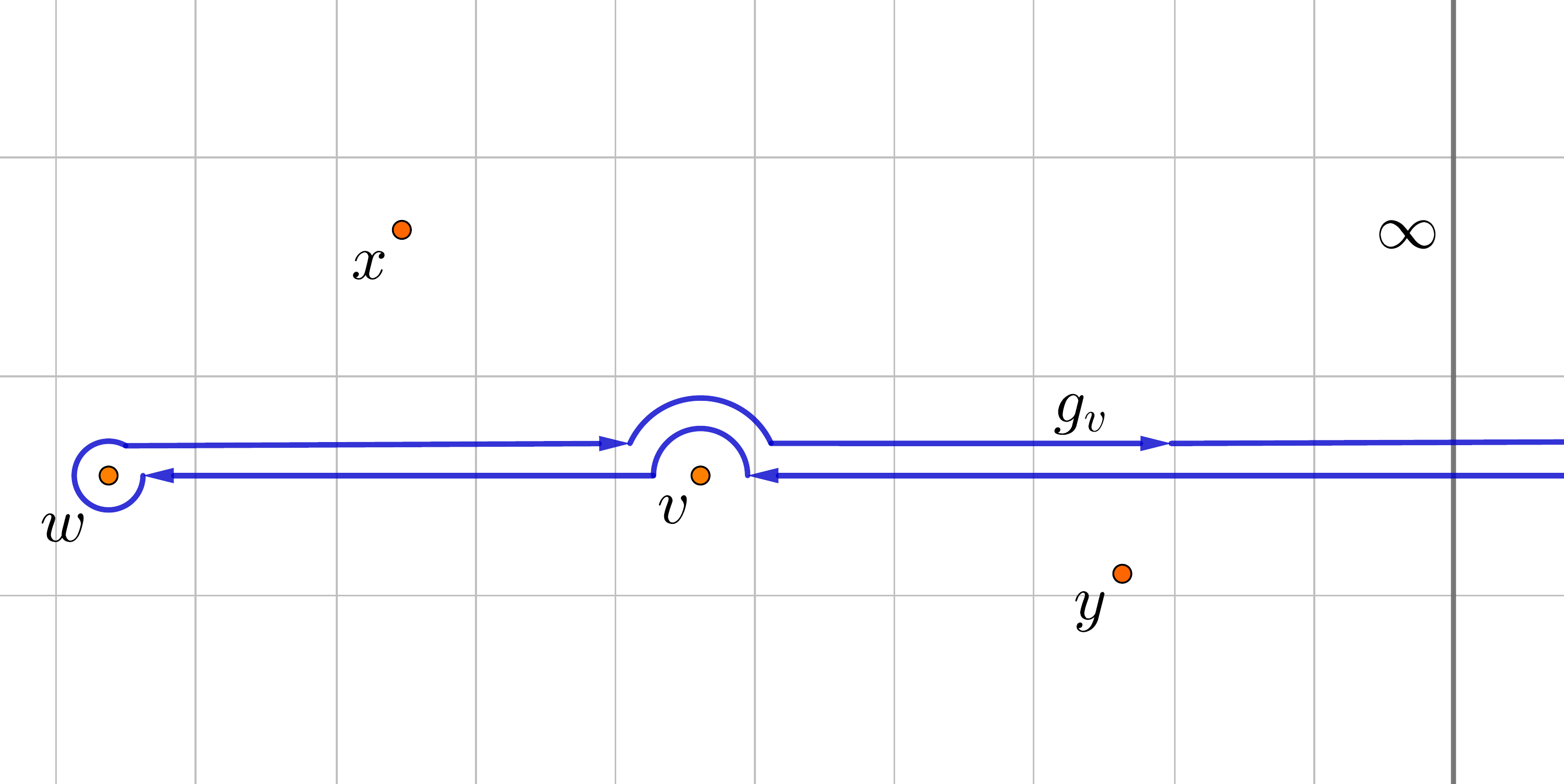}
	\caption{Example of a generator $g_w$, for a case when $\Im w=\Im v$}
	\label{pic:generator}
\end{figure}

\begin{remark}
	\label{rmrk:choice_of_generators}
	There is a freedom of choice for the generators and $\gamma_{str}$, at least in the sense of passing points of $V$ below or above, so it is just according to our convention that we pass them from above.	
\end{remark}

We finish this subsection with a few straightforward but useful lemmas.

\begin{lmm}
	\label{lmm:homotopy_in_rectangle}
	Let $\gamma\subset \mathbb{C}\setminus V$ be an r-curve with $\gamma(0)\in V$. Further, let
	$$B:=[x_1,\infty)\times[y_1,y_2]$$
	be a (right-)infinite strip so that $\gamma|_{[0,\infty)}\subset B$. Denote $V_1:=V\cap B$. Then $W(V,\gamma)\in F(V_1)$, where $F(V_1)$ is identified with its image in $F(V)$ under the natural embedding.	
\end{lmm}
\begin{proof}
	Follows from the construction of $W(V,\gamma)$.	
\end{proof}

By $\abs{W(V,\gamma)}$ we denote the length of the word $W(V,\gamma)$ in a free group. It is $\abs{W(V,\gamma)}$ that will soon allow us to measure the ``complexity'' of the homotopy information of $\id$-type points in the \tei\ space.

\begin{lmm}
	If $V_1\subset V$ and $\gamma(0)\in V_1$, then $\abs{W(V_1,\gamma)}\leq\abs{W(V,\gamma)}$.
\end{lmm}
\begin{proof}
	We can first find $W(V,\gamma)$ and then ``glue in'' the punctures from $V\setminus V_1$. Then some generators of the fundamental group in $H_r(V)$ become trivial. So the length of $W(V_1,\gamma)$ cannot be bigger than the length of $W(V,\gamma)$. 	
\end{proof}	

\begin{lmm}
	Let $\gamma\subset \mathbb{C}\setminus V$ be an r-curve with the strictly increasing real part and such that $\gamma(0)\in V$. Then $\abs{W(V,\gamma)}<C_{\abs{V}}$ where $C_{\abs{V}}>0$ depends only on the size of $V$. 
\end{lmm}
\begin{proof}
	Follows from the construction of $W(V,\gamma)$.
\end{proof}

\subsection{$H_r$ and Hawaiian earrings}

This is a minor subsection containing a rather standing aside fact that we will need later.

Let $H_r$ be defined as in Subsection~\ref{subsec:r_curves} for a countably infinite $V$. Clearly we can choose such representatives $g_{v_n}'$ of the ``straight'' generators $g_{v_n}$ (defined similarly as for finite $V$) so that each of them is a Jordan curve containing $\infty$, they are pairwise disjoint except at $\infty$, and each $g_{v_n}'$ is a concatenation $\alpha_n^{-1}\cdot\beta_n$, where $\alpha_n,\beta_n$ are horizontal paths with the strictly increasing real part, and parametrized by it.

Recall that the \emph{Hawaiian earring} is the topological space homeomorphic to $\mathcal{H}=\cup C_n$, where $C_n=\{z\in\mathbb{C}: \abs{z-1/n}=1/n\}$ for $n\in\mathbb{N}$.

\begin{lmm}[Hawaiian earring skeleton]
	\label{lmm:hawaiian_earring_skeleton}
	Let $H:=\cup g_{v_n}'$. Then $H$ is homeomorphic to $\mathcal{H}$ and is a deformation retract of $H_r$.
\end{lmm}
\begin{proof}[Sketch of the proof]
	We can construct the homeomorphism $\psi:H\to\mathcal{H}$ by mapping homeomorphically the consecutive generators $g_{v_n}'$ to the consecutive outer circles $C_n$ so that $\psi(g_{v_n}')=C_n$, $\psi(\infty)=0$, and $\psi(\alpha_n(t)),\psi(\beta_n(t))\to 0$ as $t\to\infty$ uniformly in $n$ whenever defined.
	
	Such $\psi$ is a continuous bijection. Since $H$ is compact and $\mathcal{H}$ is Hausdorff, $\psi$ is homeomorphism.
	
	The second statement of the lemma follows easily from the fact that the punctured closed unit disk deformation retracts onto its outer boundary.
\end{proof}	

\subsection{Homotopy type of a preimage}

We are now interested in how the homotopy type, and in particular $\abs{W(V,\gamma)}$, changes after taking preimage under an entire map. Although to prove Classification Theorem~\ref{thm:main_thm} we only need to consider the exponential function, we study (without essential additional efforts) the case of compositions of a polynomial with the exponential, which will be in our focus in the subsequent articles devoted to more general families. We begin with two preliminary lemmas and then provide bounds for $\abs{W(V,\tilde{\gamma})}$ of a preimage curve $\tilde{\gamma}$.

Let $V\subset \mathbb{C}$ be a \emph{finite} set, and let $\gamma$ be an r-curve such that $\gamma(0)\in V$ and $\gamma|_{\mathbb{R}^+}\subset\mathbb{C}\setminus V$. As earlier choose $H_r=H_r(V,\gamma)$. We again consider homotopy types of $\gamma$ in $H_r$.

\begin{lmm}[Piecewise linear representative]
	\label{lmm:piecewise_linear_representative}
	There exists an r-curve $\gamma_{pl}\in [\gamma]$ consisting of at most $6\abs{W(V,\gamma)}+2$ straight line segments.	
\end{lmm}
\begin{proof}
	Let $m=\abs{W(V,\gamma)}$. Take a generator $[g_w]$ of the fundamental group ${\pi_1(H_r,\infty)}$ for some $w\in V$. There exists a representative of the generator $[g_w]$ consisting of $5$ straight line segments: $1$ segment to pass from $\infty$ to a small neighborhood of $w$, $3$ to make a triangular loop around $w$ in this small neighborhood, and $1$ to return back to $\infty$.
	
	Analogously $[\gamma_{str}]$ has an representative consisting of at most $2$ straight line segments.
	
	Using the presentation of $[\gamma]=[\gamma_{str}]\cdot[\gamma_{loop}]$  and a presentation of $[\gamma_{loop}]=g_{v_1}...g_{v_m}$ via $m$ generators, we already have a piece-wise linear representative of $[\gamma]$  containing at most $5m+2$ straight line segments. Nevertheless, this representative is not an r-curve because it has some midpoints at $\infty$: the endpoints of the generators $g_{v_i}$. But this flaw can be easily adjusted by the ``price'' of at most $m$ straight line segments: if $z_1,z_2\in\mathbb{C}$ are such that $\Re z_1=\Re z_2>\max_{v\in V} \Re v$, then the concatenation
	$$\langle z_1+t,t\in[0,\infty]\rangle\cdot \langle z_2+t,t\in[0,\infty]\rangle^{-1}$$
	of horizontal rays is homotopic to a straight line segment between $z_1$ and $z_2$. This way we replace pairs of neighboring infinite rays on the generators by $\Pi$-shaped parts.
	
	Altogether this means that $[\gamma]$ has an r-curve representative consisting of at most $5m+2+m=6m+2$ straight line segments.	
\end{proof}

\begin{lmm}[Intersections bound homotopy type]
	\label{lmm:bounding_|W|}
	For every point $w \in V$, consider the ray $L_w=\langle w-t, t>0\rangle$. Assume that $\gamma|_{(0,\infty)}$ intersects $L_w$ at $k_w<\infty$ points, and let $k=\sum_{V} k_w$. Then $$\abs{W(V,\gamma)}\leq (k+1)\abs{V}.$$  
\end{lmm}
\begin{proof}
	Let $0<t_1<\dots<t_m<\infty$ be the values of $t$ for which $\gamma|_{(0,\infty)}$ intersects at least one of the $L_w$. Then $m\leq k $. It can also happen that $m<k$ since $\gamma$ might intersect a few $L_u$ at the same time.
	
	Choose $s_0 \in (0,t_1),\, s_1\in (t_1,t_2),\dots,\, s_{m-1}\in (t_{m-1},t_m),\, s_m\in (t_m,\infty)$, and define\\	
	$\gamma_0=\gamma|_{[0,s_0]},\\ \gamma_1=\gamma|_{[s_0,s_1]},\\...\\
	\gamma_m=\gamma|_{[s_{m-1},s_m]},\\\gamma_{m+1}=\gamma|_{[s_m,\infty]}$,\\	
	so we have $\gamma=\gamma_0\cdot\gamma_1\cdot...\cdot\gamma_{m+1}$.
	
	For $z\in\mathbb{C}$ denote by $\delta_z$ the infinite ray $\langle z+t,t\in [0,\infty]\rangle$. Then
	$$[\gamma]= [\gamma_0\cdot\delta_{s_0}]\cdot[\delta_{s_0}^{-1}\cdot\gamma_1\cdot\delta_{s_1}]\cdot ...\cdot[\delta_{s_{m-1}}^{-1}\cdot\gamma_m\cdot\delta_{s_m}]\cdot[\delta_{s_m}\cdot\gamma_{m+1}].$$
	
	Each $[\delta_{s_{i-1}}^{-1}\cdot\gamma_i\cdot\delta_{s_i}]$ for $i=\overline{1,m}$ is homotopic either to a consecutive product (possibly with inverses) of a few generators $g_v$ without repetitions, or it is null-homotopic.
	
	Further, $[\gamma_0\cdot\delta_{s_0}]$ is homotopic either to $\gamma_{str}$, or to a concatenation with a consecutive product (possibly with inverses) of generators $g_v$ without repetitions, whereas $[\delta_{s_m}\cdot\gamma_{m+1}]$ is evidently null-homotopic.
	
	Hence 
	$$\abs{W(V,\gamma)}\leq m\abs{V}+\abs{V}\leq (k+1)\abs{V}.$$

\end{proof}

Now we are ready to describe how $\abs{W(V,\gamma)}$ behaves after taking preimages under the exponential or a polynomial. The estimates are justified by the following elementary lemma.

\begin{lmm}[Preimages stay horizontal]
	Let $f$ be either the exponential or a monic polynomial, $\gamma:(0,\infty)\to\hat{\mathbb{C}}\setminus\{\text{singular values of }f\}$ be an r-curve, and $\tilde{\gamma}$ be one of its lifts under $f$. The following statements hold.
	\begin{enumerate}
		\item If $f$ is the exponential, then $\tilde{\gamma}$ is an r-curve.
		\item If $f$ is a monic polynomial of degree $d\geq 1$, then there exists a degree $d$ root of unity $\xi$ such that $\xi \tilde{\gamma}$ is an r-curve.
	\end{enumerate}
\end{lmm}

Thus, it makes sense to speak about how the homotopy type of an r-curve changes under lifts.

\begin{prp}[Bounds for lifts by polynomials]
	\label{prp:lift_through_polynomial}
	Let $p$ be a monic polynomial of degree $d\geq 1$ with its critical values contained in $V$, and let the set $\tilde{V}$ be such that $p(\tilde{V})\subset V$ and $p|_{\tilde{V}}$ is bijective. Denote by $\tilde{\gamma}$ the homeomorphic r-curve preimage of $\gamma$ under $p$.
	
	If $n=\abs{V}$, and $m=\abs{W(V,\gamma)}$, then $$\abs{W(\tilde{V},\tilde{\gamma})}<6dn^2(m+1).$$		
\end{prp}
\begin{proof}
	
	Choose some $H_{r'}=H_{r'}(\tilde{V},\tilde{\gamma})$ so that $p(\mathbb{H}_{r'})\supset \mathbb{H}_r$. Then every homotopy of $\gamma$ in $H_r$ lifts to a homotopy of $\tilde{\gamma}$ in $H_{r'}$.	
	
	Next, let $\gamma_{pl}\in [\gamma]$ be the piecewise linear representative consisting of at most $6m+2$ straight line segments as in the Lemma~\ref{lmm:piecewise_linear_representative} and denote by $\delta$ its lift under $p$ which is r-curve. Then $\delta$ is homotopic to $\tilde{\gamma}$ in $H_{r'}$.
	
	Elementary computation shows that a preimage of a straight line segment of $\gamma_{pl}$ under polynomial either intersects each straight line containing an $L_w$ in no more that $d$ points, or is a subset of it. If the second case takes place for some preimages of bounded segments of $\gamma_{pl}$, we slightly deform $\gamma_{pl}$ (in its homotopy class) so that the first condition holds. Note that since $\delta$ is horizontal, if the preimage of the unique infinite straight segment of $\gamma_{pl}$ is a subset of a horizontal straight line, then it does not intersect any $L_w$. Hence without loss of generality we can assume that $\delta|_{(0,\infty)}$ interests each $L_w, w\in\tilde{V}$ (see Lemma~\ref{lmm:bounding_|W|}) in at most $d(6m+2)$ points.
	
	Then due to Lemma~\ref{lmm:bounding_|W|}
	$$\abs{W(\tilde{V},\tilde{\gamma})}=\abs{W(\tilde{V},\delta)}\leq(nd(6m+2)+1)n=$$
	$$6mn^2d+2n^2d+n<6mn^2d+6n^2d=6dn^2(m+1).$$
\end{proof}

We are interested in an analogous result for the exponential.

\begin{prp}[Bounds for lifts by the exponential]
	\label{prp:lift_through_exponential}
	Let $0\in V$,\,\\ $\gamma(0)\neq 0$ and let $\tilde{V}$ be such that $\exp(\tilde{V})\subset V$, and $\exp|_{\tilde{V}}$ is injective. Denote by $\tilde{\gamma}$ the homeomorphic r-curve preimage of $\gamma$ under the exponential.
	
	If $n=\abs{V}=\abs{\tilde{V}}+1$, and $m=\abs{W(V,\gamma)}$, then $$\abs{W(\tilde{V},\tilde{\gamma})}<6n^2(m+1).$$		
\end{prp}
\begin{remark}
	Note that in this case $\abs{V}=\abs{\tilde{V}}+1$ since asymptotic value of $\exp$ does not have a preimage.	
\end{remark}
\begin{proof}
	The proof is identical to the case of a polynomial in Proposition~\ref{prp:lift_through_polynomial}, except that a preimage of a straight line segment under the exponential either intersects $L_w$ in at most \emph{one} point (instead of $d$), or is a subset of it. That is why the formula for the exponential coincides with the formula for polynomials of degree $d=1$.	
\end{proof}

At the end of this subsection we provide bounds for how $W(V,\gamma)$ changes under lifts by compositions of a polynomial and the exponential.

\begin{thm}[Bounds for lifts by $p\circ\exp$]
	\label{thm:lift_through_composition}
	Let $p$ be a monic polynomial of degree $d$, and $g=p\circ\exp$.
	
	Further, let $V\subset\mathbb{C}$ be a finite set containing singular values of $g$ and $\gamma:(0,\infty)\to\hat{\mathbb{C}}$ be an r-curve such that $\gamma(0)\subset V\setminus\{\text{singular values of }g\}$ and $\gamma|_{\mathbb{R}^+}\subset\mathbb{C}\setminus V$. 
	
	If $\tilde{V}\subset\mathbb{C}$ is a finite set such that $g(\tilde{V})\subset V\setminus\{\text{singular values of }g\}$, and $g|_{\tilde{V}}$ is injective, then for every r-curve preimage $\tilde{\gamma}$ of $\gamma$ under $g$ holds
	$$\abs{W(\tilde{V},\tilde{\gamma})}<42dn^4(m+1)$$
	where $n=\abs{V}$ and $m=\abs{W(V,\gamma)}$.	
\end{thm}
\begin{proof}
We obtain the estimate simply by using consequtively Proposition~\ref{prp:lift_through_polynomial} and Proposition~\ref{prp:lift_through_exponential} with the same $\abs{V}=n$.

$$\abs{W(\tilde{V},\tilde{\gamma})}<6n^2(6dn^2(m+1)+1)\leq 42dn^4(m+1).$$
\end{proof}
 
\subsection{Spiders}
\label{subsec:spiders}

Now return to our usual setting with the captured exponential function $f:=c\circ f_0$ defined as in Subsection~\ref{subsec:capture}. Moreover, we assume that the external address of $a_0$ is not (pre-)periodic.

\begin{defn}[Spider]
	\label{defn:spider}
	An image $S_{\varphi}$ of the standard spider $S_0$ under an $\id$-type map $\varphi$ is called a \emph{spider}.	
\end{defn}

The notion of a spider comes with a few natural follow-up definitions.

\begin{defn}[Spider legs]
	The image of a ray tail $R_n$ under a spider map is called a \emph{leg} (of a spider).
\end{defn}

\begin{defn}[Subspider]
	For a spider $S_\varphi$, a non-empty union of its legs is called a \emph{subspider} of $S_\varphi$.
\end{defn}

We also want to have equality relation on the set of spiders.

\begin{defn}[Equal spiders]
	We say that the spiders $S_\varphi$ and $S_\psi$ are equal if for all $n\geq 0$ we have $\varphi(R_n)=\psi(R_n)$ (as point sets).
\end{defn}

\begin{lmm}[Equal spiders define the same point in $\mathcal{T}_f$]
	\label{lmm:equal_spiders}
	Let $S_\varphi=S_\psi$. Then $\varphi$ is isotopic to $\psi$ relative $S_0$, and consequently $[\varphi]=[\psi]$.	
\end{lmm}
\begin{proof}
	We might assume that $\varphi|_{S_0}=\psi|_{S_0}$, otherwise just isotope $\varphi$ in disjoint neighborhoods of the legs $\{\varphi(R_n)\}$ so that the legs get reparametrized. Consider the map $\psi\circ\varphi^{-1}$. It is equal to identity on the locally connected boundary of the simply-connected domain $\mathbb{C}\setminus\varphi(S_0)$. Hence $\psi\circ\varphi^{-1}|_{\mathbb{C}\setminus\varphi(S_0)}$ is isotopic to identity relative $\varphi(S_0)$. This implies that $[\varphi]=[\psi]$.  	
\end{proof}

Denote $\mathcal{O}_n=\{a_i\in P_f:i\leq n\}$.

\begin{defn}[Leg homotopy word]
	Let $S_\varphi$ be a spider. Then the \emph{leg homotopy word} $W_n^\varphi$ of $\varphi(R_n)$ is 
	$$W_n^\varphi:=W(\varphi(\mathcal{O}_n),\varphi(R_n)).$$
\end{defn}

Next theorem helps to estimate how the leg homotopy words change under Thurston iteration.

\begin{thm}[Combinatorics of preimage]
	\label{thm:homotopy_type_under_pullback}
	Let $\varphi$ be \idt. Then $$\abs{W_n^{\hat{\varphi}}}<A(n+2)^4 \max\{1,\abs{W_{n+1}^\varphi}\},$$ where $A$ is a positive real number.
\end{thm}
\begin{proof}
	This is just a restatement of Theorem~\ref{thm:lift_through_composition} in a particular context. That is, when $V=\varphi(\mathcal{O}_{n+1})$, $\gamma=\varphi(R_{n+1})$ and $W(V,\gamma)=W_{n+1}^\varphi$. Then if we take $\tilde{\gamma}=\hat{\varphi}(R_n)$,
	$$\abs{W_n^{\hat{\varphi}}}<42(n+2)^4 (\abs{W_{n+1}^\varphi}+1)<A(n+2)^4\max\{\abs{W_{n+1}^\varphi},1\},$$
	where $A>0$ is a positive real number.
\end{proof}

\subsection{Spiders and \tei\ equivalence}

In the last subsection we prove the results showing that the homotopy type of $\id$-type maps can be uniquely encoded using spiders. More precisely, for a map $\varphi$ \idt\ one can uniquely recover its equivalence class $[\varphi]\in\mathcal{T}_f$ from the positions of points in $\varphi(P_f)$ and $W_n^\varphi$ for all $n\geq 0$. 

\begin{defn}[Equivalence of spiders]
	We say that two spiders $S_\varphi$ and $S_\psi$ are \emph{equivalent} if for all $n\geq 0$ we have $\varphi(a_{n})=\psi(a_{n})$ and legs with the same index are homotopic, i.e.\ $[\varphi(R_{n})]=[\psi(R_{n})]$ in $\mathbb{C}\setminus \varphi(P_f)$.
\end{defn}

\begin{prp}[Spider equivalence is Teichm\"uller equivalence]
	\label{prp:spiders_define_teich_point}
	Two spiders $S_\varphi$ and $S_\psi$ are equivalent if and only if $[\varphi]=[\psi]$, i.e.\ $\varphi$ is isotopic to $\psi$ relative ${P_f}$.
\end{prp}
\begin{proof}	
	$(\Leftarrow)$ ``If'' direction follows directly from the definitions.
	
	$(\Rightarrow)$ Assume that the spiders $S_\varphi$ and $S_\psi$ are equivalent.
	
	We need to prove that $S_\varphi$ can be isotoped into $S_\psi$ via an ambient isotopy of $\mathbb{C}\setminus P_f$. The statement of the proposition will then follow due to Lemma~\ref{lmm:equal_spiders}. To simplify the notation we assume that $\varphi=\id$, i.e.\ $S_\varphi=S_0$ (in the general case the proof is exactly the same). 
	
	The proof uses the same idea as Alexander's trick, where we shrink all ``entanglements'' into one point.
	
	First, note that one can introduce a natural linear order on the set of the legs of $S_0$: for every pair of legs $R_i$ and $R_j$ we can say which one is \emph{higher} or \emph{lower}. More precisely, choose a right half-plane $\mathbb{H}_r$ containing both $R_i$ and $R_j$ and join $a_i$, the endpoint of $R_i$, to a finite point of $\partial\mathbb{H}_r$ via an injective path $\gamma$ inside of $\mathbb{H}_r\setminus(R_i\cup R_j)$. The union of two curves $\gamma$ and $R_i$ divides $\mathbb{H}_r$ into two parts. We say that $R_j$ is \emph{higher} (resp. \emph{lower}) than $R_i$ if it is contained inside of the upper (resp. lower) part. 
	
	For every $n\geq 0$ there exists a minimal number $k_n\geq 1$ so that $S_0\setminus\cup_{i=0}^n R_i$ is a disjoint union of $k_n$ subspiders $S_n^j$, $j\in\{1,...,k_n\}$ such that for every $j\in\{1,...,k_n\}$ and $i\leq n$ all legs of the subspider $S_n^j$ are simultaneously higher or lower that $R_i$. Clearly, $k_n\leq n+2$.
	
	Let $U$ be a right half-plane containing $S_0$. Now define inductively a sequence of shrinking Jordan domains contained in $U$.
	\begin{itemize}
		\item \textbf{(Step 0)} Define $k_0$ mutually disjoint Jordan domains $U_0^j$ so that
		$$S_0^j\subset U_0^j\subset U$$
		and the diameter of $U_0^j$ in the spherical metric is less than the diameter of $S_0^j$ plus $1$.
		\item \textbf{(Step n)} Define $k_n$ mutually disjoint Jordan domains $U_n^j$ so that for some $i\leq k_{n-1}$
		$$S_n^j\subset U_n^j\subset U_{n-1}^i$$
		and the diameter of $U_n^j$ in the spherical metric is less than the diameter of $S_n^j$ plus $1/(n+1)$.
	\end{itemize}	
	
	We are going to construct the isotopy $\psi_u^\infty$ which is the concatenation of countably many ``local'' isotopies $\psi_u^n, n\geq 0$. Roughly speaking, via $\psi_u^0$ we isotope $\psi(R_0)$ into $R_0$, then via $\psi_u^1$ we isotope $\psi_u^0(R_1)$ relative $P_f$ into $R_1$ without moving $R_0$,..., then via $\psi_u^n$ we isotope $\psi_u^{n-1}(R_n)$ into $R_n$ without moving $R_0, R_1,...,R_{n-1}$, and so on. 
	
	We construct this sequence of isotopies inductively. Without loss of generality assume that $\psi(S_0)\subset U$ and $\psi|_{\mathbb{C}\setminus U}=\id$. Note that every leg $\psi(R_n)$ is homotopic to $R_n$ inside of $U\setminus P_f$ (consequence of the fact that $\psi$ is \idt\ ).
	\begin{itemize}
		\item \textbf{(Step 0)} Let $\psi_u^0, u\in[0,1]$ be an isotopy of $\hat{\mathbb{C}}\setminus P_f$ (or equivalently of $\mathbb{C}\setminus P_f$) so that
		\begin{enumerate}
			\item $\psi_0^0=\psi$,
			\item $\psi_1^0|=\id$ on $\mathbb{C}\setminus \cup_{j=1}^{k_0}U_0^j$,
			\item $\psi_u^0\equiv\id$ on $\mathbb{C}\setminus U$.
		\end{enumerate} 
		\item \textbf{(Step n)} Let $\psi_u^n, u\in[0,1]$ be an isotopy of $\hat{\mathbb{C}}\setminus P_f$ so that
		\begin{enumerate}
			\item $\psi_0^n=\psi_1^{n-1}$,
			\item $\psi_1^n|=\id$ on $\mathbb{C}\setminus \cup_{j=1}^{k_n}U_n^j$,
			\item $\psi_u^n\equiv\id$ on $\mathbb{C}\setminus \cup_{j=1}^{k_{n-1}}U_{n-1}^j$.
		\end{enumerate} 
	\end{itemize}	
	Note that every Step $n$ is well-defined since the preceding Step $n-1$ guarantees that the leg $\psi(R_n)$ is contained inside of some $U_{n-1}^j$ and is homotopic to $R_n$ inside of $U_{n-1}^j\setminus P_f$.
	
	In order to define the concatenation of these infinitely many isotopies we repara-metrize them so that $\psi_u^n$ is defined on the interval $[1-1/2^n,1-1/2^{n+1}]$. As $n\to\infty$, maps $\varphi_u^n$ are not equal to identity only in neighborhoods of $\infty$ that are getting smaller. Hence the infinite concatenation $\psi_u^\infty=\psi_u^0\cdot\psi_u^1\cdot...$ is well defined. This finishes the proof of the proposition. 	
\end{proof}

\begin{defn}[Projective equivalence of spiders]
	\label{defn:proj_equiv}
	We say that two spiders $S_\varphi$ and $S_\psi$ are projectively equivalent if for all $n\geq 0$ we have $\varphi(a_n)=\psi(a_n)$ and $W_n^\varphi=W_n^\psi$.
\end{defn}

\begin{prp}[Projectively equivalent spiders are equivalent]
	\label{prp:proj_spiders_define_teich_point}
	Two spiders $S_\varphi$ and $S_\psi$ are \emph{projectively equivalent} if and only if they are equivalent.
\end{prp}
\begin{proof}
	$(\Leftarrow)$  If $\varphi(R_n)$ and $\psi(R_n)$ are homotopic relative to $\psi(P_f)$, then they are homotopic relative any subset of $\psi(P_f)$. 
	
	$(\Rightarrow)$ Assume now that the spiders $S_\varphi$ and $S_\psi$ are projectively equivalent. Let $r\in\mathbb{R}$ be such that $\mathbb{H}_r$ contains both $S_\varphi$ and $S_\psi$, and consider $H_r=H_r(\varphi(P_f), \varphi(R_0))$ (defined as before Lemma~\ref{lmm:bar_paths_equivalence}).
	
	We want to prove that the spider legs $\varphi(R_n)$ and $\psi(R_n)$ belong to the same homotopy class in $H_r$ (the condition $W_n^{\varphi}=W_n^\psi$ says only that $\varphi(R_n)$ and $\psi(R_n)$ are homotopic in $H_r$ relative to only a finite set of punctures $\varphi(\{a_i\}_{i=0}^n)$).
	
	We claim that the homotopy type of each $\varphi(R_n)$ in $H_r$ is uniquely defined by the knowledge of $W_k^\varphi$ for all $k\geq n$. 
	
	Pick $k\geq n$. Assume that we know the homotopy type of $\varphi(R_n)$ in the ``finitely-punctured'' $H_r\cup \varphi(P_f\setminus\{a_i\}_{i=0}^k)$ (for $k=n$ we have it via $W_n^\varphi$) and know $W_{k+1}^\varphi$. Since $R_{k+1}\cap R_n=\emptyset$ and $\varphi$ is \idt, the homotopy type of $\varphi(R_n)$ in $H_r\cup \varphi(P_f\setminus\{a_i\}_{i=0}^{k+1})$ is uniquely defined.
	
	Next, from Lemma~\ref{lmm:hawaiian_earring_skeleton} we know that $H_r$ deformation retracts onto a subspace, homeomorphic to the Hawaiian earring $\mathcal{H}=\cup C_n$, so that the ``straight'' generator $g_{v_n}$ of $\pi_1(H_r,\infty)$ is mapped to an element of $\pi_1(\mathcal{H},0)$ represented by the circle $C_n$. It is well-known fact about Hawaiian earrings (e.g.\ see \cite{hawaiian}) that $\pi_1(\mathcal{H},0)$ is a subgroup of $\varprojlim F_j$, where $F_j$ is the fundamental group of $\cup_1^j C_i$ with the base point at $0$, considered as a subgroup of $\pi_1(\mathcal{H},0)$ after collapsing all other $C_i$ to $0$. Hence, if $\theta_1,\theta_2\in\pi_1(\mathcal{H},0)$ are such that $\theta_1=\theta_2$ in $F_j$ for all $j$, then $\theta_1=\theta_2$ in $\pi_1(\mathcal{H},0)$. From this follows that the homotopy type of $\varphi(R_n)$ in $H_r$ is uniquely determined by its homotopy types in every $H_r\cup \varphi(P_f\setminus\{a_i\}_{i=0}^k)$ for $k\geq n$.	
\end{proof}

From Proposition~\ref{prp:spiders_define_teich_point} and Proposition~\ref{prp:proj_spiders_define_teich_point} follows the theorem that identifies the projective equivalence of spiders and \tei\ equivalence.

\begin{thm}[Projective equivalence of spiders is \tei\ equivalence]
	\label{thm:W_define_teich_point}
	Two spiders $S_\varphi$ and $S_\psi$ are projectively equivalent if and only if $[\varphi]=[\psi]$, i.e.\ $\varphi$ is isotopic to $\psi$ relative ${P_f}$.
\end{thm}

\begin{remark}
	Note that the only properties of $S_0$ that we used in the proofs of Lemma~\ref{lmm:equal_spiders} and Propositions~\ref{prp:spiders_define_teich_point}, \ref{prp:proj_spiders_define_teich_point} are that $S_0$ is contained in some right half-plane and all $R_n$ are mutually disjoint (except at $\infty$) injective r-curves with the spherical diameters tending to $0$ as $n\to\infty$. Hence the statements are obviously true for every such more general choice of $S_0$ (and after replacement of $P_f$ by the set of endpoints of $S_0$).
\end{remark}

\section{Invariant compact subset}

\label{sec:inv_compact_subset}

In this section we construct a compact invariant subset $\mathcal{C}_f$ of the \tei\ space $\mathcal{T}_f$. In the first theorem we present the construction and prove invariance. A statement that $\mathcal{C}_f$ is compact will be the content of the second theorem.

Let $f=c\circ f_0$ be the captured exponential function constructed in Subsection~\ref{subsec:capture}, and for $n\geq 0$ let $t_n$ be the potential of the point $a_n\in P_f$ (then $t_{n+1}=F(t_n)$).

Define the set $\mathcal{P}':=\{\frac{t_{n+1}+t_n}{2}: n\geq 0\}$ and if $\rho=\frac{t_{n+1}+t_n}{2}\in\mathcal{P}'$, let $N(\rho):=n$. Further, define $D_\rho:=\mathbb{D}_\rho(0)$ and $M_\rho:=e^{2t_{N(\rho)}}$.

The subset $\mathcal{C}_f$ is constructed in the form of a list of conditions.

\begin{thm}[Invariant subset]
	\label{thm:invariant_subset}
	Fix some $\rho\in\mathcal{P}'$ and denote $N:=N(\rho)$. Let $\mathcal{C}_f(\rho)\subset\mathcal{T}_f$ be the \emph{closure} of the set of points in $\mathcal{T}_f$ represented by $\id$-type maps $\varphi$ for which there exists an isotopy $\varphi_u, u\in[0,1]$ \idt\ maps such that $\varphi_0=\id,\varphi_1=\varphi$, and the following conditions are simultaneously satisfied.
	\begin{enumerate}		
		\item (Marked points stay inside of $D_\rho$) If $n\leq N$,
		$$\varphi_u(a_n)\in D_\rho.$$
		\item (Precise asymptotics outside of $D_\rho$) If $n>N$, then
		$$\abs{\varphi_u(a_n)-a_n}<1/n.$$		
		\item (Separation inside of $D_\rho$) If $k<l\leq N$, then $$\abs{\varphi_u(a_k)-\varphi_u(a_l)}>\frac{1}{(M_\rho)^{N-l+1}}.$$
		\item (Bounded homotopy) If $n\leq N$, then 		
		$$\abs{W_n^{\varphi_u}}<A^{N-n+1}\left(\frac{(N+2)!}{(n+1)!}\right)^4$$		
		where $A$ is the constant from Theorem~\ref{thm:homotopy_type_under_pullback}.\\
	\end{enumerate}
	
	Then if $\rho\in\mathcal{P}'$ is big enough, the subset $\mathcal{C}_f(\rho)$ is well-defined, invariant under the $\sigma$-map and contains $[\id]$.	
\end{thm}

Let us briefly discuss the content of the theorem before proving it. Our invariant subset (and compact, as will be seen later) is described as a closure of a set of $\id$-type maps $\varphi$ satisfying conditions (1)-(4).

Conditions (1)-(2) say that the maps $\varphi$ have to be ``uniformly \idt'', that is, the marked points outside of a disk $D_\rho$ have precise asymptotics, while inside of $D_\rho$ we allow some more freedom.

Condition (3) tells us that the marked points inside of $D_\rho$ cannot come very close to each other --- this is necessary to keep our set bounded in the \tei\ space. Moreover, it is needed to control the distance to the asymptotic value --- if a marked point is too close to it, then after Thurston iteration its preimage has the real part close to $-\infty$, and this spoils condition (1).

Condition (4) takes care of the homotopy information and provides bounds for leg homotopy words of points inside of $D_\rho$. Note that similar bounds for marked points outside of $D_\rho$ are encoded implicitly in condition (2).

\begin{proof}[Proof of Theorem~\ref{thm:invariant_subset}]
	Let $\mathcal{C}_f^\circ(\rho)\subset\mathcal{C}_f(\rho)$ be the set of points in $\mathcal{T}_f$ of which we take the closure in the statement of the theorem, i.e.\ represented by $\id$-type maps $\varphi$ for which there exists an isotopy $\varphi_u$ \idt\ maps such that $\varphi_0=\id,\varphi_1=\varphi$ and the conditions (1)-(4) are simultaneously satisfied. Since $\sigma$-map is continuous, it is enough to prove invariance of $\mathcal{C}_f^\circ(\rho)$ for big $\rho$.
	
	First, note that for big $\rho$ the set $\mathcal{C}_f^\circ(\rho)$ contains $[c^{-1}]$: $c^{-1}$ can be joined to identity via the isotopy $c_u^{-1}$ where $c_u$ was constructed in Subsection~\ref{subsec:capture}.
	
	Further, from the asymptotic formula~\ref{eqn:as_formula} and the fact that $\Re a_n/\Im a_n\to\infty$ follows that for all $\rho\in\mathcal{P}'$ big enough the first $N+1$ point $a_0,a_1,...,a_N$ of the orbit $P_f$ are contained in $D_\rho$, while the other points are outside and mutual distances between them and from them to $D_\rho$ are bigger than one. In other words, conditions (1)-(2) allow the first $N+1$ points to move under isotopy $\varphi_u$ only inside of $D_\rho$, while for every $n>N$ the point $a_n$ moves inside of a disk $D_n$ of radius $1/n$, and all these disks $D_\rho,D_{N+1},D_{N+2},...$ are mutually disjoint with the mutual distance between them bigger than one. Moreover, we can assume that for $n>N$ we have $\Re a_n+1<\Re a_{n+1}$.
	
	For all $\varphi\in\mathcal{C}_f^\circ(\rho)$ after concatenation with $c_u^{-1}$ we obtain the isotopy $\psi_u$ \idt\ maps with $\psi_0=c^{-1},\psi_1=\varphi$ and satisfying conditions (1)-(4).  Then $\hat{\psi}_u$ is an isotopy \idt\ maps with $\hat{\psi}_1=\id$. Let $g_u(z)=\psi_u\circ f\circ\hat{\psi}_u^{-1}(z)=e^z+\kappa_u$. Now we want to prove that $\hat{\psi}_u$ satisfies each of the items (1)-(4): from this would follow that $\hat{\varphi}\in\mathcal{C}_f^\circ(\rho)$. 
	
	We prove that each of the conditions (1)-(4) for $\hat{\varphi}$ follows from the conditions (1)-(4) for $\varphi$.
	
	\textbf{(4)} Note that since for $n>N$ we have $\Re a_n+1<\Re a_{n+1}$, for $n>N$ we have $$\abs{W_N^{\psi_u}}=0<1.$$
	
	Hence from Theorem~\ref{thm:homotopy_type_under_pullback} for $n\leq N$ we get $$\abs{W_n^{\hat{\psi}_u}}<A(n+2)^4\max\{\abs{W_{n+1}^{\psi_u}},1\}<A(n+2)^4 A^{N-(n+1)+1}\left(\frac{(N+2)!}{(n+2)!}\right)^4=$$
	$$A^{N-n+1}\left(\frac{(N+2)!}{(n+1)!}\right)^4.$$
	
	\textbf{(3)} We use the formula $\psi_u(a_{n+1})=e^{\hat{\psi}_u(a_n)}+\kappa_u$. Since $\rho=\frac{t_N+t_{N-1}}{2}$ and $\kappa_u\in D_\rho$, for $n\leq N$ and big $\rho$ we have
	$$\Re \hat{\psi}_u(a_n)=\log\abs{\psi_u(a_{n+1})-\kappa_u}<\log\abs{t_{N+1}+1+\rho}<$$
	$$\log(2t_{N+1})=\log(2F(t_N))<2t_N.$$
	
	Thus for $k<l\leq N$ and big enough $\rho$ 
	$$\abs{\psi_u(a_{k+1})-\psi_u(a_{l+1})}=\abs{\int_{\hat{\psi}_u(a_k)}^{\hat{\psi}_u(a_l)}e^z dz}\leq\int_{\hat{\psi}_u(a_k)}^{\hat{\psi}_u(a_l)}\abs{e^z} \abs{dz}<M_\rho\abs{\hat{\psi}_u(a_k)-\hat{\psi}_u(a_l)},$$
	and
	$$\abs{\hat{\psi}_u(a_k)-\hat{\psi}_u(a_l)}>\frac{\abs{\psi_u(a_{k+1})-\psi_u(a_{l+1})}}{M_\rho}>\frac{1}{(M_\rho)^{N-l+1}}.$$
	
	\textbf{(2)} Since $\hat{\psi}_u$ is continuous in $u$ and $\hat{\psi}_0=\id$, for $n>N$
	$$\hat{\psi}_u(a_n)=\log(\psi_u(a_{n+1})-\kappa_u)=\log (a_{n+1}+O(1)-\kappa_u)=$$
	$$a_n+\log\left(\frac{a_{n+1}+O(1)-\kappa_u}{e^{a_n}}\right)=a_n+\log\left(1+\frac{\kappa_0+O(1)-\kappa_u}{a_{n+1}-\kappa_0}\right).$$
	
	If $\rho$ (and consequently $N$) is big enough, we have $\abs{\hat{\psi}_u(a_n)-a_n}<1/n$ for all $n>N$ (due to the very fast growth of $\abs{a_n}$).
	
	\textbf{(1)} Instead of proving directly that for $n\leq N$ holds $\hat{\psi}_u(a_n)\in D_\rho$, we prove that for big $\rho$
	$$-\frac{\rho}{2}<\Re\hat{\psi}_u(a_n)<\frac{\rho}{2}$$
	and
	$$-\frac{\rho}{2}<\Im\hat{\psi}_u(a_n)<\frac{\rho}{2}.$$
	From this would follow that $\abs{\hat{\psi}_u(a_n)}<\rho/{\sqrt{2}}<\rho.$
	
	In the proof of (3) we have shown that for big $\rho$ holds $\Re \hat{\psi}_u(a_n)<2t_N$. Clearly, $2t_N<\rho/2$ for big $\rho$.
	
	Using condition (3) for $\psi_u$ we get
	$$\Re \hat{\psi}_u(a_n)=\log\abs{\psi_u(a_{n+1})-\kappa_u}>\log\frac{1}{M_\rho^n}=-2nt_N$$
	which is bigger than $-\rho/2$ for all $\rho$ big enough.
	
	Since $\abs{W_{n+1}^{\psi_u}}<A^{N-n}\left(\frac{(N+2)!}{(n+2)!}\right)^4$, the spider leg $\hat{\psi}_u(R_{n+1})$ makes no more than $A^{N-n}\left(\frac{(N+2)!}{(n+2)!}\right)^4+1$ loops around the singular value. Hence the difference between $2\pi s_n$ and the imaginary part of $\hat{\psi}_u(a_n)$ will be less than $2\pi \left(A^{N-n}\left(\frac{(N+2)!}{(n+2)!}\right)^4+1\right)$. Due to fast growth of the sequence $\{t_n\}_{n=0}^\infty$, by making $\rho$ bigger we may assume that for all $n\leq N$ we have 
	$$2\pi \left(A^{N-n}\left(\frac{(N+2)!}{(n+2)!}\right)^4+1+\abs{s_n}\right)<\frac{\rho}{2}.$$
	Hence $\abs{\Im\hat{\psi}_u(a_n)}<\rho/2$ and $\abs{\hat{\psi}_u(a_n)}<\rho/\sqrt{2}<\rho$.	
\end{proof}

\begin{remark}
	From the construction of $\mathcal{C}_f(\rho)$ immediately follows that the elements of $\mathcal{C}_f(\rho)$ are asymptotically conformal.	
\end{remark}

Next, we prove another key statement claiming compactness of $\mathcal{C}_f(\rho)$. Note that it is very natural to expect this compactness because of the conditions (1)-(4) of the Theorem~\ref{thm:invariant_subset}. Indeed, note that (1)-(4) are formulated in terms of how the marked points $a_n$ are allowed to move under isotopies $\varphi_u$ \idt\ maps: the points inside of $D_\rho$ have to stay inside, points outside of $D_\rho$ move only slightly. At the same time the points $a_n$ inside of $D_\rho$ are allowed to rotate around points in $\mathcal{O}_n$ only some bounded amount of times. Outside of $D_\rho$ no ``rotations'' are allowed and basically nothing happens. The final remark is that the marked points are not allowed to come too close to each other.

\begin{thm}[Compactness]
	\label{thm:compact_subset}
	For all $\rho\in\mathcal{P}'$ big enough $\mathcal{C}_f(\rho)$ is a compact subset of $\mathcal{T}_f$.	
\end{thm}
\begin{proof}
	We are going to prove that the set $\mathcal{C}_f^\circ(\rho)$ from the proof of Theorem~\ref{thm:invariant_subset} is pre-compact, or equivalently, since $\mathcal{T}_f$ is a metric space in the \tei\ metric, that every sequence $\{[\varphi^n]\}\subset\mathcal{C}_f^\circ(\rho)$ has a subsequence that converges to some point $[\varphi]\in\mathcal{T}_f$.
		
	Since each $[\varphi^n]\in\mathcal{C}_f^\circ(\rho)$, we can assume that every $\varphi^n$ is \idt\ and $\varphi_u^n$ is an isotopy with $\varphi_0^n=\id, \varphi_1^n=\varphi^n$ satisfying conditions (1)-(4) of Theorem~\ref{thm:invariant_subset}.
	
	Recall from the proof of Theorem~\ref{thm:invariant_subset} that for all $\rho\in\mathcal{P}'$ big enough the first $N+1$ marked points move under isotopy $\varphi_u$ only inside of $D_\rho$, while for every $n>N$ the point $a_n$ moves inside of a disk $D_n$ of radius $1/n$, and all these disks $D_\rho,D_{N+1},D_{N+2},...$ are mutually disjoint with the mutual distance between them bigger than one. Hence we can assume that when $u\in[0,1/2]$, we have $\varphi^n_u|_{D_\rho}=\id$ for every $n$, and when $u\in[1/2,1]$, we have $\varphi^n_u|_{\cup_{k={N+1}}^\infty D_k}=\id$ for every $n$, that is, first the marked points inside of $D_\rho$ do not move, and afterwards do not move the marked point outside of $D_\rho$. This means that we simply need to prove the theorem in two separate cases: when only the marked points outside of $D_\rho$ move, and when only the marked points inside of $D_\rho$ move. 
	
	In the former case the sequence $[\varphi^n]$ clearly has a limit point in $\mathcal{T}_f$, since marked points move inside of mutually disjoint small disks with mutual distance between them bounded from below.
	
	Assume now that the marked points outside of $D_\rho$ do not move. We are looking for a convergent subsequence in $\{\varphi^n\}$. Let $V:=P_f\cap D_\rho$.
	
	We say that $\varphi^k(V)$ and $\varphi^l(V)$ have the same configuration if for every pair of points $a_m,a_n\in V$ we have
	$$\Re\varphi^k(a_m)<\Re\varphi^k(a_n)\iff \Re\varphi^l(a_m)<\Re\varphi^l(a_n),$$
	$$\Re\varphi^k(a_m)=\Re\varphi^k(a_n)\iff \Re\varphi^l(a_m)=\Re\varphi^l(a_n),$$
	$$\Im\varphi^k(a_m)<\Im\varphi^k(a_n)\iff \Im\varphi^l(a_m)<\Im\varphi^l(a_n),$$
	$$\Im\varphi^k(a_m)=\Im\varphi^k(a_n)\iff \Im\varphi^l(a_m)=\Im\varphi^l(a_n).$$
	
	Since $V$ is finite, only finitely many different configurations are possible. After passing to a subsequence we may assume that all $\varphi^k(a_n)$ converge to some points $b_n$, $W_n^{\varphi^k}=\const$ for all $a_n\in V$,  and all $\varphi^k(V)$ have the same configuration.
	
	Note that for $k$ big enough $\varphi^k(a_n)$ will be compactly contained in disjoint small disks $D_n$ contained in $\mathbb{D}_{\rho+1}(0)$ with centers at the limits $b_n$, so we can assume this holds for all $n$.
	
	Now let $I_n^k:[0,1]\to\mathbb{C}$ such that $$I_n^k(u)=\varphi^k(a_n)+u(\varphi^{k+1}(a_n)-\varphi^k(a_n)),$$ and let $\xi_u^k$ to be an isotopy of \qc\ maps such that it is identity outside of $\cup D_n$, $\xi_0^k=\id$ and $\xi_u^k(\varphi^k(a_n))=I_n^k(u)$ for all $a_n\in V$.
	
	Note that $[\varphi^{k+1}]=[\xi_1^k\circ\varphi^k]$. Indeed, while moving simultaneously along $I_n^k$, the sets $\xi_1^k\circ\psi^k(V)$ do not change their configuration. It is easy to see that in this case $W_n^{\xi_u^k\circ\varphi^k}=W_n^{\varphi^k}$. But then from Theorem~\ref{thm:W_define_teich_point} follows that $\xi_1^k\circ\varphi^k$ and $\varphi^{k+1}$ define the same point in $\mathcal{T}_f$.
	
	Hence $[\varphi^k]=[\xi_u^0\cdot\xi_u^1\cdot...\cdot\xi_u^k]$ converges to some $[\varphi]$ in the \tei\ metric. This is an easy consequence of the fact that \tei\ space of the unit disk with one marked point is homeomorphic to the unit disk.
\end{proof}

\section{Proof of Classification Theorem}

In this section we are going to finally prove the Classification Theorem~\ref{thm:main_thm}. Before the proof we state a variation of the Banach Fixed Point Theorem that will be needed for this.

\begin{lmm}[Adjusted Banach Fixed Point Theorem]
	\label{lmm:Banach_fixed_point}
	Let $X$ be a compact complete metric space, and $\sigma:X\to X$ be a strictly contracting map. Then $X$ contains a fixed point of $\sigma$ and it is unique.
\end{lmm}
\begin{proof}
	The proof is essentially the same as for the classical Banach Fixed Point Theorem.
	
	Assume that $X$ does not contain a fixed point of $\sigma$ and define a map $\kappa:X\to\mathbb{R}$ so that $\kappa(x)=d(\sigma(x),\sigma^2(x))/d(x,\sigma(x))$ where $d(x,y)$ is the metric function of $X$. Since $X$ is compact, $\kappa$ attains its maximum on $X$. Moreover, because $\sigma$ is strictly contracting, this maximum is less than $1$. But this means that for every $x\in X$ the sequence $x,\sigma(x),\sigma^2(x),...$ is Cauchy and due to completeness converges to some point in $X$. This point will be a fixed point of $\sigma$. Uniqueness is straightforward from the strict contraction of $\sigma$.
\end{proof}

Now, we prove Theorem~\ref{thm:main_thm}.

\begin{proof}[Proof of Classification Theorem~\ref{thm:main_thm}]
	
	We show existence and uniqueness of $g$ separately.
	
	\textbf{(Existence.)} Choose as some $f_0\in\mathcal{N}$ so that its singular value does not escape. Then due to Theorem~\ref{thm:as_formula} in $I(f_0)$ there is a point $z$ that escapes on rays with potential $t$ and external address $\underline{s}$.
	
	Let $f=c\circ f_0$ be the captured exponential function constructed as in Subsection~\ref{subsec:capture} and having singular value $z$ escaping as under $f_0$.
	
	From Theorems~\ref{thm:invariant_subset} and \ref{thm:compact_subset} we know that there is a non-empty compact invariant under $\sigma$ subset $\mathcal{C}_f=\mathcal{C}_f(\rho)\subset\mathcal{T}_f$ such that all its elements are asymptotically conformal. Next, from Theorem~\ref{thm:sigma_strictly_contracting} follows that $\sigma$ is strictly contracting in the \tei\ metric on $\mathcal{C}_f$.
	
	Thus we have a strictly contracting map $\sigma$ on a compact complete metric space $\mathcal{C}_f$. Note that it is \emph{not} true that $\sigma$ is uniformly strictly contracting on $\mathcal{C}_f$ (think of the map $x^2$ on the interval $[0,1/2]$), so we cannot apply the Banach Fixed Point Theorem directly. But from Lemma~\ref{lmm:Banach_fixed_point} follows that $\mathcal{C}_f$ contains a fixed point $[\varphi_0]$. Note (even though we do not use this fact later) that, because $\sigma$ is strictly contracting on $\mathcal{C}_f$, the constructed fixed point $[\varphi_0]$ is the unique fixed point in $\mathcal{C}_f$.
	
	Now let $g=\varphi_0\circ f\circ\hat{\varphi}_0^{-1}$. This is an entire function that is Thurston equivalent to $f$. Moreover, since every point in $\mathcal{C}_f$ is \idt, without loss of generality we can assume that $\varphi_0$ is \idt.
	
	Now we need to prove that $\varphi_0(a_0)$ escapes on rays with potential $t$ and external address $\underline{s}$.
	
	Note first that since the map $\varphi_0$ is \idt,  the singular value $\varphi_0(a_0)$ escapes on rays. Indeed, the dynamic ray of $g$ having external address $\underline{s}$ is defined for potentials $t>t_{\underline{s}}$ (because none of its images contains the singular value), and the point $p$ of this ray having potential $t>t_{\underline{s}}$ has the same asymptotics under iterations of $g$ as $\varphi_0(a_0)$. Since $g^{-1}$ is uniformly contracting on its tracts near $\infty$, we see that for some integer $N>0$ the point $g^N(p)$ coincides with $g^N(\varphi_0(a_0))$. For the same reason we can assume that for $n\geq N$ the legs $\varphi_0(R_n)$ coincide with the ray tails of $g$ towards $g^n(\varphi_0(a_0))$ (just isotope $\varphi_0$: every such $\varphi_0(R_n)$ is in the same homotopy class as the ray tail). Since the preimages of $\varphi_0(R_n)$ recover according to the external address, $\varphi_0(R_0)$ coincides with the ray tail of $g$ towards $\varphi_0(a_0)$.
	
	\textbf{(Uniqueness.)} Assume that there exists another function $h\in\mathcal{N}$ satisfying the same conditions on the combinatorics and speed of escape of the singular value. Let $S_0^g$ and $S_0^h$ be the standard spiders of $g$ and $h$, respectively (i.e.\ union of the ray tails towards the singular orbit).
	
	Since $S_0^g$ and $S_0^h$ have the same vertical order of legs, one can map $P_g$ into $P_h$ with a map $\chi$ \idt\ belonging to an asymptotically conformal point of $\mathcal{T}_g$ so that $\chi(S_0^g)$ is equivalent to $S_0^h$ as spiders. But then we can define the map $\sigma:\mathcal{T}_g\to\mathcal{T}_g$ for the function $g$ (i.e\ with the capture identically equal to identity), and then the map $\hat{\chi}$ will be isotopic to $\chi$ relative $P_g$ --- this is an immediate consequence of the fact that $\chi$ and $\hat{\chi}$ are \idt\ and that $\chi$ can be lifted by the branched covering maps $g$ and $h$. On the other hand, since $[\id]$ is invariant under $\sigma$ ($g$ is entire), and $[\chi]$ is asymptotically conformal, either the \tei\ distance from $[\chi]$ to $[\id]$ must be strictly bigger than the \tei\ distance from $[\hat{\chi}]$ to $[\id]$ (Theorem~\ref{thm:sigma_strictly_contracting}) or $[\chi]=[\id]$ and $g=h$. Since $[\chi]$ is a fixed point of $\sigma$, only the latter option is possible. 
\end{proof}

\section{Acknowledgements}

We would like to express our gratitude to our research team in Aix-Marseille Universit\'e, especially to Dierk Schleicher who supported this project from the very beginning, Sergey Shemyakov who carefully proofread all drafts, as well as to  Kostiantyn Drach, Mikhail Hlushchanka, Bernhard Reinke and Roman Chernov for uncountably many enjoyable and enlightening discussions of this project at different stages. We also want to thank Dzmitry Dudko for his multiple suggestions that helped to advance the project, Lasse Rempe for his long list of comments and relevant questions, and Adam Epstein for important discussions especially in the early stages of this project. 

Finally, we are grateful to funding by the Deutsche Forschungsgemeinschaft DFG, and the ERC with the Advanced Grant “Hologram” (695621), whose support provided excellent conditions for the development of this research project.

\vspace{1em}
Aix-Marseille Universit\'e, France

\textit{Email:} bconstantine20@gmail.com

\end{document}